\documentclass{amsart}

\RequirePackage[OT1]{fontenc}
\RequirePackage{verbatim, amscd,epsfig,amsmath,amsthm,amstext,amssymb,tikz}
\usepackage{graphicx}
\usepackage{url}
\usepackage{mathrsfs}
\usepackage[all]{xy}
\usepackage{lineno}
\usepackage{tikz}
\usetikzlibrary{patterns}
\usepackage{caption}
\usepackage{subcaption}
\usepackage{tikz-cd}
\usepackage{bbm}

\renewcommand{\d}{\mbox{dist}}

\newcommand{\cS}{\mathcal{S}}
\newcommand{\cR}{\mathcal{R}}

\newcommand{\ECT}{\operatorname{ECT}}
\newcommand{\LECT}{\operatorname{LECT}}
\newcommand{\SELECT}{\operatorname{SELECT}}

\newcommand{\N}{\mathbb{ N}}

\newcommand\define[1]{\textbf{#1}}

\newcommand{\Omin}{\mathcal{O}}
\newcommand{\CF}{\operatorname{CF}}
\newcommand{\Def}{\operatorname{Def}}
\newcommand{\Rad}{\mathcal{R}}
\newcommand{\CS}{\operatorname{CS}}
\newcommand{\AffGr}{\text{AffGr}}

\theoremstyle{plain}
\newtheorem{thm}{Theorem}[section]
\newtheorem{lem}[thm]{Lemma}
\newtheorem{prop}[thm]{Proposition}
\newtheorem{cor}[thm]{Corollary}

\newtheorem{defn}[thm]{Definition}
\newtheorem{ex}[thm]{Example}
\newtheorem{rmk-defn}[thm]{Remark}
\newtheorem{rmk-thm}[thm]{Remark}
\newtheorem{rmk-lem}[thm]{Remark}
\newtheorem{rmk-prop}[thm]{Remark}

\numberwithin{equation}{section}

\title[The Lifted Euler Characteristic Transform]{Representing Fields without Correspondences: the Lifted Euler Characteristic Transform}

\author{Henry Kirveslahti}
\address{Department of Statistical Science, Duke University; Durham, NC USA}
\email{henry.kirveslahti@duke.edu}

\author{Sayan Mukherjee}
\address{Departments of Statistical Science, Mathematics, Computer Science, Biostatistics \& Bioinformatics, Duke University; Durham, NC USA}
\email{sayan@stat.duke.edu}

\subjclass[2021]{Primary: 62R40, 52C45; Secondary: 68U05}
\keywords{Euler calculus, persistent homology, statistical shape analysis}

\begin{document}

\begin{abstract}

Topological transforms have been very useful in statistical analysis of shapes or surfaces without restrictions that the shapes are diffeomorphic and requiring the estimation of correspondence maps. In this paper we introduce two topological transforms that generalize from shapes to fields,  $f:\mathbb{R}^3 \rightarrow \mathbb{R}$. Both transforms take a field and associate to each direction $v\in S^{d-1}$ a summary obtained by scanning the field 
in the direction $v$. The transforms we introduce are of interest for both applications as well as their theoretical properties. The topological transforms for shapes are based on an Euler calculus on sets. A key insight in this paper is that via a lifting argument one can develop an Euler calculus on real valued functions from the standard Euler calculus on sets, this idea is at the heart of the two transforms we introduce. We prove
the transforms are injective maps. We show for particular moduli spaces of functions we can upper bound the number of directions needed determine any particular function.
\end{abstract}

\maketitle

\section{Introduction}
In this paper we introduce topological transforms for fields $f: \mathbb{R}^d \rightarrow \mathbb{R}$ that are both practically useful 
for data analysis and are of theoretical interest. In particular we introduce the Lifted Euler Characteristic Transform (LECT) and the 
Super Lifted Euler Characteristic Transform (SELECT). Previously, topological transforms---the Euler Characteristic Transform 
(ECT) and Persistent Homology Transform (PHT)---were used to analyze classes of shapes without restricting the shapes to be
diffeomorphic and requiring correspondence maps between shapes \cite{PHT}. The generalization to shape classes that are not
diffeomorphic is important because in several datasets \cite{MillerNotices} the shapes are not qualitatively or topologically the same and the PHT and ECT approaches were one of the first that could be applied to these datasets. In the setting of modeling and comparing data that are
represented as three-dimensional fields state-of-the-art analysis tools are based on diffeomorphism or deformation based approaches 
\cite{Ashburner}. In the same way that the ECT and PHT allowed for the analysis of shape collections that are not diffeomorphic the LECT and SELECT allow for the analysis of classes of fields that are not diffeomorphic.

The extension from shapes to fields is of interest, because not all imaging modalities fit into the shape framework. Examples of such modalities include magnetic resonance imaging (MRI), micro-computed tomography (Micro-CT) for soft tissue, and positron emission tomography (PET). These are all inherently continuous objects, while the construction behind the ECT and PHT is discrete and their underlying machinery does not work with continuous data without reservations \cite{definableintegrals}.

In both data science and computational geometry, quantifying differences in fields is a difficult problem. The crux of the problem is formulating a digital representation of the field that allows
for effective and efficient downstream analysis and preserves all the information in the field. A fundamental problem in representing fields as well as shapes is that there is no obvious coordinate system. The naive representation of a field as function values on a three dimensional grid works poorly in applications as the field fundamentally carries geometric information, again this is the reason that diffeomorphism based methods are used. We will show that using the SELECT we can both cluster fields as well as predict clinical signatures from the shape of the field, a regression
problem. These applications are analogous to the shape setting where the PCT and ECT were used to cluster heel bones from primates for evolutionary
applications \cite{PHT} and predict clinical outcomes of brain tumors based on their shape in~\cite{tumor}.

Before we outline the LECT and SELECT we first summarize the PHT and ECT. At a high level the ECT and PHT consider a shape which is a subset $M$ of $\mathbb{R}^d$, and associate to each direction $v\in S^{d-1}$ a shape summary obtained by scanning $M$ in the direction $v$. This process of scanning has a Morse-theoretic and persistent-topological flavor---we study the topology of the sublevel sets of the height function $h_v=\langle v,\cdot\rangle\mid_M$ as the height varies. This yields what is called the \define{Euler Curve} of $M$ in direction $v$. In the case of the ECT, this curve records the Euler characteristic of each sublevel set, and for the PHT, the \define{persistence diagram}, which pairs critical values of $h_v$ in a computational way. The ECT and the PHT associate to any sufficiently tame subset $M\subset \mathbb{R}^d$ a map from the sphere to the space of Euler curves and persistence diagrams, respectively.

The main theoretical challenge in extending topological summaries such as the ECT or PHT from shapes to fields is extending Euler calculus from sets to real valued functions. In Section \ref{sec:lift} we introduce the Lifted Euler Characteristic Transform which uses a lifting argument to 
define an Euler calculus that holds for real valued functions and allows us to define the analog of the ECT for fields. In Section \ref{sec:inj}
we prove that the LECT is an injective map when one considers all directions of the sphere. We then introduce a variant of the LECT called the Super Lifted Euler Characteristic Transform which is much more robust and designed for real applications in Section \ref{sec:robust}. In Section
\ref{sec:mar} we relate the Euler calculus we proposed based on the lifting idea to other approaches to Euler calculus for real valued functions.
In Section \ref{sec:howmany} we provide an upper bound on the number of directions required to determine a function for certain moduli spaces of functions. These moduli spaces are uncountable and infinite so this result is very interesting as it states a finite representation
can identify functions in the classes. In Section \ref{sec:applications} we demonstrate the promise of the SELECT on simulated and real data.

\section{An Euler Calculus and Euler Characteristic Transform for Fields}

At the heart of the ECT and PHT is the idea of integrating constructible and definable functions \cite{tametopology} with respect to the Euler characteristic as a finitely-additive measure, this idea is called Euler calculus \cite{definableintegrals}. Recall that the ECT was designed for shapes or surfaces. Our goal in this section is to develop an analog for the ECT for fields, specifically $f: \mathbb{R}^d \rightarrow \mathbb{R}$. 

In this section we will propose the Lifted Euler Characteristic Transform which is based on an Euler calculus for fields as an analog to the ECT for shapes.

\subsection{Review of the Euler Characteristic Transform}
\label{sec:rev}

We first review the ECT and the tameness conditions we require for spaces of shapes. Here we provide a summary of the ideas, for details see
\cite{CMT}. An important theme in this paper is our generalization of the ECT and Euler calculus for fields will follow from the same tameness 
conditions we state here.

In imaging applications shapes are digitized as meshes so it is natural to consider a simplicial complex as the mathematical representation of a shape. In \cite{CMT} a more general mathematical characterization of tameness properties for shapes for data analysis to be well posed was developed. The tameness properties required are abstract conditions on a general class of sets called o-minimal structures \cite{tametopology}.

\begin{defn}
An o-minimal structure $\mathcal{O}= \{\mathcal{O}_d\}$, $d\ge 0$ is a collection of subsets $\mathcal{O}_d$ of $\mathbb{R}^d$ which are closed under intersection and complement. The collection satisfies the following axioms: 
\begin{itemize}
	\item If $A \in \mathcal{O}_d$, then $A\times \mathbb{R}$ and $\mathbb{R}\times A$ are in $\mathcal{O}_{d+1}$.
	\item If $A \in \mathcal{O}_{d+1}$, then $\pi(A)\in \mathcal{O}_d$ where $\pi: \mathbb{R}^{d+1}\to \mathbb{R}^d$ is an axis aligned projection. 
\end{itemize}
In addition, $\mathcal{O}$ contains all semi-algebraic sets and $\mathcal{O}_1$ contains only finite unions of points and open intervals in $\mathbb{R}$. Elements of $\mathcal{O}$ are called definable or tame sets.
\end{defn}

Intuitively, an o-minimal set is akin to a simplicial complex, and the latter is an important example of the former. O-minimal structures are very general and include semi-algebraic, sub-analytic, and piecewise linear sets. We now state a general class of sets based on o-minimal structures for which transforms such as the ECT and PHT have a mathematically rigorous foundation. 

\begin{defn}(Constructible sets)
A constructible set is a compact definable subset of $\mathbb{R}^d$. The set of constructible sets are denoted as CS($\mathbb{R}^d$). 
\end{defn}

Definable sets play the role of measurable sets for an integration theory based on the Euler characteristic in Euler calculus.
The guarantee that definable sets can be measured by the Euler characteristic is by virtue of the following theorem which also
connects directly to both meshes, the digitized representation of a shape, and the more abstract notion of constructible sets.

\begin{thm}[Triangulation Theorem~\cite{tametopology}]\label{thm:triangulation}
Any tame set admits a definable bijection with a subcollection of open simplices in the geometric realization of a finite Euclidean simplicial complex.
Moreover, this bijection can be made to respect a partition of a tame set into tame subsets.
\end{thm}

Here a definable bijection is a map that is one-to-one and onto and whose graph is a definable set. The partition can be chosen to be the simplices in the geometric realization. The upshot of this choice is that the Triangulation Theorem allows us to define the Euler characteristic of a tame set in terms of an alternating count of the number of simplices used in a definable triangulation:

\begin{defn}
\label{def24}
If $X\in \Omin$ is tame and $h:X\to \cup \sigma_i$ is a definable bijection with a collection of open simplices, then the Euler characteristic of $X$ is
\[
\chi(X):=\sum_i (-1)^{\dim \sigma_{i}}
\]
where $\dim \sigma_i$ denotes the dimension of the open simplex $\sigma_i$.
We understand that $\chi(\varnothing)=0$ since this corresponds to the empty sum. 
\end{defn}
\begin{ex}
Let $M$ be a simplicial surface embedded in $\mathbb{R}^3$, which is a two-dimensional triangulated mesh. According to Definition \ref{def24}, we get a well-known identity for the Euler Characteristic $\chi(M)$ of $M$: 
$$
\chi(M) = \# V(M) - \#E(M) + \#T(M),
$$
where $\# V(M), \#E(M), \#T(M)$ are the number of vertices, edges and triangles of $M$, respectively.
\end{ex}

The Euler characteristic is well defined on any o-minimal structure. The integration theory called Euler calculus is well-defined for constructible functions. For developing the theory, we will move from constructible sets to constructible functions, which are defined as follows:

\begin{defn}
\label{def26}
A constructible function $\phi:X \to \mathbb{Z}$ is an integer-valued function on a definable set $X$ with the property that every level set is definable and only finitely many level sets are non-empty.
The set of constructible functions with domain $X$, denoted $\CF(X)$, is closed under pointwise addition and multiplication, thereby making $\CF(X)$ into a ring.
\end{defn}

We are now able to define the Euler integral which the ECT is based on.
\begin{defn}
\label{def27}
The Euler integral of a constructible function $\phi:X \to \mathbb{Z}$ is the sum of the Euler characteristics of each of its level sets:
\[
\int \phi \, d\chi := \sum_{n=-\infty}^{\infty} n \cdot \chi(\phi^{-1}(n)).
\]
Note that constructibility of $\phi$ implies finitely many of the $\phi^{-1}(n)$ are non-empty. 
\end{defn}

\begin{ex}
Let $M$ again be a simplicial surface embedded in $\mathbb{R}^3$. Denote by $1_M$ the indicator function of $M$, which is a constructible function defined as:
$$
1_M(x)=\begin{cases}
			1, & \text{if $x \in M$}\\
            0, & \text{otherwise.}
		 \end{cases}
$$
By comparing the Definitions \ref{def24} and \ref{def27}, we readily see that the Euler characteristic of the constructible function $1_M$ agrees with that of the definable set $M$.
\end{ex}

We can now define the ECT based on the Euler integral.

\begin{defn}
\label{def29}
The Euler Characteristic Transform takes a constructible function $\phi$ on $\mathbb{R}^d$ and returns a constructible function on 
$S^{d-1}\times \mathbb{R}$ whose value at a direction $v$ and real parameter $t\in\mathbb{R}$ is the Euler integral of the restriction of $\phi$ to the half space $x\cdot v \leq t$:
\[
\ECT: \CF(\mathbb{R}^d) \to \CF(S^{d-1} \times \mathbb{R}) \quad \text{where} \quad 
\ECT(\phi)(v,t):=\int_{x\cdot v \leq t} \phi \, d\chi.
\]
\end{defn}
Definition \ref{def29} can be intuitively explained as follows: We consider a scanning direction $v$ and record the Euler characteristic of the function (e.g., the indicator function of a shape) up to a certain height $t$. The value that the ECT takes at $(v,t)$ is equal to this recorded Euler characteristic. This value can be computed easily by taking an alternating sum of the simplices.

For $M$ a constructible subset, we write $\ECT(M)$ as a shorthand for $\ECT(1_M)$. As $M$ is a constructible set, the indicator function $1_M$ is a constructible function so that this is a well defined map from $\CS(\mathbb{R}^d)$, the set of constructible sets on $\mathbb{R}^d$, to the set $\CF(S^{d-1} \times \mathbb{R})$ of constructible functions on $S^{d-1} \times \mathbb{R}$. When we fix $v\in S^{d-1}$ and let $t\in\mathbb{R}$ vary, we refer to $f_{M}^{v}(t):=\ECT(M)(v,-)$ as the Euler curve in direction $v$.
This allows us to equivalently view the Euler Characteristic Transform for a fixed $M\in \CS(\mathbb{R}^d)$ as a map from the sphere to the space of Euler curves:
\[\ECT(M): S^{d-1} \to \CF(\mathbb{R})  \quad \text{where} \quad 
v \mapsto f_{M}^{v}(t) :=\chi(M\cap \{x \mid x\cdot v \leq t\}).
\]

A very attractive theoretical property of the ECT is that it is an injective map \cite{CMT,RG-ECT-inject}. The following
is Theorem 3.5 in \cite{CMT}.
\begin{thm}[Curry-Mukherjee-Turner]
\label{thm210}
Let $\CS(\mathbb{R}^d)$ be the set of constructible sets.
The map $\ECT: \CS(\mathbb{R}^d) \to \CF(S^{d-1} \times \mathbb{R})$ is injective.
If $M$ and $M'$ are two constructible sets for which one obtains the same Euler curves for each direction then they are the same set:
\[
\ECT(M)=\ECT(M'): S^{d-1} \to \CF(\mathbb{R}) \Longrightarrow M=M'.
\]
\end{thm}
The applied upshot of Theorem \ref{thm210} is that no information is lost by considering the ECT of a mesh instead of the original shape.

The key mathematical idea in proving the above theorem is applying an inversion theorem of Schapira~\cite{Schapira:tom} which regularizes the ill-posed problem of inverting the Radon transform based on tameness as characterized by o-minimal structures. To make the paper self-contained
we state Schapira's inversion formula in Appendix \ref{secA1}.

\subsection{The Lifted Euler Characteristic Transform}
\label{sec:lift}

The mathematical theory presented in the previous section applies to constructible functions, i.e., functions that take finitely many values. This model is well-suited for hard tissue, where the shape has a clear-cut boundary -- a point either belongs to the shape or it does not. However, in other cases, we may be interested in a continuum of values with varying density, which can be better represtend as field data.

In this section, we explain how the ECT approach for analyzing shapes can be extended to field-type data using a lifting argument. This allows us to analyze new types of data while retaining the desirable properties of the ECT. The mathematical model for our fields consists of definable functions $\Def(\mathbb{R}^d)$, and we are particularly interested in maps from a compact subset $X \subset \mathbb{R}^d$ to a compact subset $Y \subset \mathbb{R}$. The transform we introduce is called the Lifted Euler Characteristic Transform (LECT).

From a theoretical perspective, a notable property of the LECT is that all the theory and methodology from Euler calculus based on o-minimal structures transfers from sets to fields. This includes an inversion formula based on Schapira's inversion theorem, which proves the injectivity of the LECT, making it an invertible transform.

A classic idea in studying the topology of random fields was characterizing the superlevel sets of random fields \cite{RandomFields}. Specifically, given a three-dimensional random field one can filter or threshold on the superlevel sets of the field values $Y \geq t$ and compute the Euler characteristic for the thresholded field, $X: Y(X) \geq t$. This analysis tool has been used extensively as a null model for various imaging modalities, including  magnetic resonance imaging (MRI) data \cite{RandomFields}. Our lifting approach is heavily motivated by this idea of extrema of random fields.

The following definition will be central to extending the ECT and Euler calculus from sets to fields. This is similar to how we represent shape data as constructible functions with Definition \ref{def26}. The idea is that given a map 
$f:\mathbb{R}^3 \rightarrow \mathbb{R}$, if the graph of the map is a definable set, then so are the domain and image of $f$.

\begin{defn}
For definable sets $X \subseteq \mathbb{R}^n$ and $Y \subseteq \mathbb{R}^k$ the function $f: X \rightarrow Y$ is definable if the graph of the function 
$f : X \rightarrow Y$ is a definable subset of $\mathbb{R}^n \times  \mathbb{R}^k$. We denote the set of definable functions from $\mathbb{R}^d$ to $\mathbb{R}$ as $\Def(\mathbb{R}^d)$.
\end{defn}

Next we define the Lifted Euler Characteristic Transform. 
\begin{defn}
For $f: \mathbb{R}^d \rightarrow \mathbb{R}$ a definable function, the \emph{lift of the ECT} along $f$ is a map
\begin{equation*}
\LECT: \Def(\mathbb{R}^d) \rightarrow \CF(S^{d-1} \times \mathbb{R} \times \mathbb{R})    
\end{equation*}

defined as
\begin{equation*}
\LECT(f)(v,h,t) = \chi \Big( \{x \in \mathbb{R}^{d} \ \mid \ x \cdot v \le h , \, f(x) = t \} \Big). 
\end{equation*}
\end{defn}

Euler integrals and their use in analyzing continuous type data have been studied previously. One related construction is the weighted Euler Curve Transform \cite{WECT}, which considers a weighted Euler characteristic. Positive weights are associated to the open simplices of a simplicial complex to indicate their importance, making use of the ring structure of the constructible functions. However, a true continuous extension, i.e., an extension to the set of definable functions, seems more evasive. These limits were studied in \cite{definableintegrals}. As noted in that paper, there are obstructions to using it for integral transforms. In particular, such transforms may not be invertible, unlike the lifted transforms presented here as we will show in Subsection \ref{sec:inj}. These ideas have been developed further in a slightly different context in \cite{leb21}, where the authors study combinations of Euler and Lebesgue integrals.

There is also a more subtle reason why one may find the Lifted Transform preferable over weighting. If one sees the field as carrying
geometric information rather than simply weighting simplices, then the differences between two fields should not be the same as the difference of their functions. This is analogous to representing shapes with indicator functions; for example, while two nested convex bodies may be similar, the difference between their indicator functions is a sphere.

Additionally, by lifting the transform we obtain a stratified map. This is a key property for shape reconstruction and has been used for feature selection in applications \cite{SINATRA}. This will be explained more in Subsection \ref{sec:stratification}.

\subsubsection{The Lifted Euler Characteristic Transform is Injective}
\label{sec:inj}

The two main utilities of the ECT are the fact that the map is injective from $\CF(\mathbb{R}^d) \to \CF(S^{d-1} \times \mathbb{R})$ and that the representation of the transformed shape is convenient for downstream statistical analysis. We now prove the LECT is also injective. This follows the proof of Theorem 3.5 in \cite{CMT}.

\begin{thm}\label{thm:LECT-injects} 
Let $\Def(\mathbb{R}^d)$ be the set of definable functions $\mathbb{R}^d \rightarrow \mathbb{R}$.
The map $\LECT: \Def(\mathbb{R}^d) \to \CF(S^{d-1} \times \mathbb{R} \times \mathbb{R})$ is injective.
Equivalently, if $f$ and $g$ are two definable functions that determine the same association of directions and level sets to Euler curves, then they are, in fact, the same function.
Said symbolically:
\[
\LECT(f)=\LECT(g): S^{d-1} \to \CF(\mathbb{R} \times \mathbb{R})  \Longrightarrow  f=g.
\]
\end{thm}
\begin{proof}

Let $f$ be a definable function $f: \mathbb{R}^d \rightarrow \mathbb{R}$.
We first prove that the transform determines an arbitrary $t$-level set $M_t$ of $f$. As $f$ definable, its level sets are constructible subsets of $\mathbb{R}^d$. Write $M_t$ for the $t$-level set of $f$, and note that by definition $\LECT(f)(v,h,t)=\ECT(M_t)(v,h)$, so this result follows directly from Theorem 3.5 in \cite{CMT} which we restate below.
Let $W$ be the hyperplane defined by $\{x\cdot v = h\}$ for $x \in M_t$.
By the inclusion-exclusion property of the definable Euler characteristic

\begin{eqnarray*}
\chi(M_t \cap W) &=& \chi(\{ x\in M_t: x\cdot v=h\})\\
& = &\chi(\{ x\in M_t: x\cdot v\leq h\} \cap \{ x\in M_t: x\cdot (-v) \leq -h\} )\\
& = &\chi(\{ x\in M_t: x\cdot v\leq h\}) + \chi( \{ x\in M_t: x\cdot (-v) \leq -h\} ) - \chi(M_t)\\
& = &\ECT(M_t)(v,h) + \ECT(M_t)(-v,-h) - \ECT(M_t)(v)(\infty).
\end{eqnarray*}
This means that from $\ECT(M_t)$ we can deduce the $\chi(M_t \cap W)$ for all hyperplanes $W\in \AffGr_d$.

Let $S$ be the subset of $\mathbb{R}^d \times \AffGr_d$ where $(x,W)\in S$ when $x$ is in the hyperplane $W$. 
For simplicity, we denote the projection to $\mathbb{R}^d$ by $\pi_1$ and the projection to $\AffGr_d$ by $\pi_2$.
For this choice of $S$ the Radon transform of the indicator function $1_{M_t}$ for $M_t \in \Omin_d$ at $W\in \AffGr_d$ is 
\begin{eqnarray*}
(\Rad_{S}1_{M_t})(W)&=& (\pi_2)_* [(\pi_1^*1_{M_t})1_S](W) \\
&=&\int_{(x,W)\in S} (\pi_1^* 1_{M_t}) \, d\chi\\
&=&\int_{x\in M_t \cap W} (\pi_1^* 1_{M_t}) \, d\chi\\
&=&\chi(M_t \cap W).
\end{eqnarray*}
This implies that from $\ECT(M_t)$ we can derive $\cR_\cS(1_{M_t})$.

Similarly let $S'$ be the subset of $\AffGr_d \times \mathbb{R}^d$ where $(W,x)\in S'$ when $x$ is in the hyperplane $W$.
For all $x\in \mathbb{R}^d$, $S_x\cap S'_x$ is the set of hyperplanes that go through $x$ and hence is $S_x\cap S'_x=\mathbb{R} P^{d-1}$ and $\chi(S_x\cap S'_x)=\frac{1}{2}(1+(-1)^{d-1})$.  
For all $x\neq x'\in \mathbb{R}^d$, $S_x\cap S'_{x'}$ is the set of hyperplanes that go through $x$ and $x'$ and hence is $S_x\cap S'_{x'}=\mathbb{R} P^{d-2}$ and $\chi(S_x\cap S'_{x'})=\frac{1}{2}(1+(-1)^{d-2})$.  
Applying Theorem 2.13 in \cite{CMT}
$$(\Rad_{S'}\circ \Rad_{S})(1_{M_t})=(-1)^{d-1} 1_{M_t} + \frac{1}{2}(1+(-1)^{d-2}) \chi(M_t)1_{\mathbb{R}^d}.$$

Note that if $\ECT(M_t)=\ECT(M'_t)$ then $\Rad_{S} 1_{M_t}=\Rad_{S} 1_{M'_t}$, since $\ECT$ determines the Euler characteristic of every slice.
Moreover, if $\ECT(M_t)=\ECT(M'_t)$, then $\chi(M_t)=\chi(M'_t)$, which by inspecting the inversion formula above further implies that $1_{M_t}=1_{M'_t}$ and hence $M_t=M'_t$. 

Applying the above reconstruction result for all $t$ it holds that for two definable functions $f,g$ if $\LECT(f) = \LECT(g)$ then $M_t(g) = M_t(f)$ for all level sets $t$.

Recall that two functions $f$ and $g$ are equal, if their domain and co-domain sets are the same and their output values agree on the whole domain,
or $f = g$ if  $f(x) = g(x)$ for all $x \in X.$ We have already proven for two definable functions $f=g$ that if $\LECT(M(f)) = \LECT(M(g))$ then
$M_t(f) = M_t(f)$ or the graph of the function for all values $f(x)=t$ are equivalent. As the graphs are equivalent the functions are equivalent.
\end{proof}

The above injectivity result can be thought of as an inversion formula for a topological Radon transform of continuous data. Essentially, the Lifted Transform is a topological Radon Transform in codimension 1, where the fiber is trivial, and the Super Lifted Transfom has a contractible fiber.

\subsubsection{Robustness and Statistical Considerations}
\label{sec:robust}

A basic result of Morse theory is that a Morse function gives a handlebody decomposition of a manifold. It only takes one critical point to completely change the topology. The level sets of two fields can look very different even if their level sets are mostly the same. However, this is not true for superlevel sets. From this viewpoint, we define the Super Lifted Euler Characteristic Transform (SELECT) as a robust adaptation of the LECT.

\begin{defn}
For $f: \mathbb{R}^d \rightarrow \mathbb{R}$ a definable function with compact support, the \emph{Super Lifted Euler Characteristic Transform} (SELECT) of $f$ is a map
\begin{equation*}
\SELECT: \Def(\mathbb{R}^d) \rightarrow \CF(S^{d-1} \times \mathbb{R} \times \mathbb{R})    
\end{equation*}

defined as
\begin{equation*}
\SELECT(f)(v,h,t) = \chi \big(\{x \in \mathbb{R}^{d} \ \mid \ x \cdot v \le h , \, f(x) \geq t \} \big). 
\end{equation*}
\end{defn}

The Super Lifted Transform is also injective:
\begin{thm}\label{thm:SELECT-injects}
The map $\SELECT: \Def(\mathbb{R}^d) \to \CF(S^{d-1} \times \mathbb{R} \times \mathbb{R})$ is injective.
\[
\SELECT(f)=\SELECT(g): S^{d-1} \to \CF(\mathbb{R} \times \mathbb{R}) \Longrightarrow  f=g.
\]\end{thm} 
\begin{proof}
As the proof is almost identical to that of Theorem \ref{thm:LECT-injects}, we omit it.
\end{proof}

We will use the SELECT to measure the dissimilarity between fields. Given two fields $f,g$ and the transforms
$\tilde{f} = \SELECT(f), \, \tilde{g} = \SELECT(g)$ the distance between $f$ and $g$ is
\begin{equation*}
\mbox{dist}(f,g)_p = \Big( \int_{S^{d-1}} \int_{\mathbb{R}} \int_{\mathbb{R}} \bigm\vert \tilde{f}(v,h,t)-\tilde{g}(v,h,t) \bigm\vert^p \ dt \ dh \ dv \Big)^{1/p}.
\end{equation*}

For shapes, this distance agrees with the ECT distance stated in \cite{CMT}. In this case, the parameter $t$ takes just two values, 0 and 1, and for constructible shapes $X$ and $Y$ and their indicator functions $1_X$ and $1_Y$ we have:

\begin{align*}
\mbox{dist}(1_X,1_Y)_p & = \Big( \int_{S^{d-1}} \int_{\mathbb{R}} \int_{\mathbb{R}} \bigm\vert \tilde{f}(v,h,t)-\tilde{g}(v,h,t) \bigm\vert^p \ dt \ dh \ dv \Big)^{1/p} \\
& = 1 \cdot \Big( \int_{S^{d-1}} \int_{\mathbb{R}} \bigm\vert \tilde{f}(v,h,1)-\tilde{g}(v,h,1) \bigm\vert^p \ dh \ dv \Big)^{1/p} \\
& = \mbox{dist}_{ECT}(X,Y)_p. 
\end{align*}

This distance is also well defined when $t$ takes more general values. In that case, we may alternatively choose to define another distance obtained by first integrating out $t$. This gives an extension of weighted Euler curves.  We will talk more about this distance in Subsection \ref{sec:mar}.

\subsection{Stratifications}
\label{sec:stratification}
One of the appealing properties of the ECT is that the transform is a stratified map, i.e. locally constant in $S^{d-1} \times \mathbb{R}$. This property is very useful for inverse problems such as reconstruction, and was one of the main ingredients in \cite{SINATRA} where it was used for pulling back statistical evidence on the shapes. In this short subsection we show that both the lifted transforms share this property with the ECT. This a direct consequence of the results in \cite{o-min-strat}. The association
\begin{equation*}
\mathbb{R}^d \times \mathbb{R} \supset M \times \mathbb{R} \hookrightarrow X_M, 
\end{equation*}
where $X_M \subset \mathbb{R}^d \times \mathbb{R} \times S^{d-1} \times \mathbb{R} \times \mathbb{R}$,
produces a definable $X_M$ for a definable $M$ and the projection $\pi$ to the last 3 symbols
\begin{equation*}
\pi: \mathbb{R}^d \times \mathbb{R} \times S^{d-1} \times \mathbb{R} \times \mathbb{R} \rightarrow S^{d-1} \times \mathbb{R} \times \mathbb{R}  
\end{equation*}
is a definable map. Figure \ref{Dparabola2} illustrates the stratification on a simple example on two points and a constructible distance metric.

\begin{ex}
\label{exs}
Consider a point cloud $X =\{p_1, \ldots p_n \} \in \mathbb{R}^d$ and let 

\begin{equation*}
f: \mathbb{R}^d \rightarrow \mathbb{R}: x \mapsto -\min_{ i \in [n]} d(x-p_i),
\end{equation*}
for some constructible metric $d$, such as the Euclidean metric. As the minimum of definable functions is definable, we may apply SELECT
to this dataset.

Whenever $x \cdot v \ge h+t$, the 2-dimensional persistence is constant in $h$, and the persistence module in $t$ just records the homology of the \v{C}ech complex associated to point cloud $X$. A simple example of the transform is depicted in Figure \ref{Dparabola2}.

The superlevel set $f^{-1}(0)$ is just $X$ and for $t>0$ the superlevel sets take the form

\begin{equation*}
f^{-1}([0,t]) = \cup_{i=1}^{n} \{ S_i(t) \},
\end{equation*}
where $S_i(t)$ is a ball of radius $t$ centered at $p_i$. The $S_i(t)$'s are constructible sets and hence so is $f^{-1}(t)$. The top dimensional strata of the stratification on $S^{d-1} \times \mathbb{R} \times \mathbb{R}$ is determined by the boolean algebra of the nerve complex of observability of $\mathcal{S}= \{S_1,\ldots, S_n \}$. In the case of the Euclidean norm, the observability function $O_i$ of an individual $S_i$ is given by the Heaviside step function
\begin{equation*}
O_i(h,v,t) = \mathbbm{H}_\theta(h), \quad \theta = \big(\|p\|-t \big) \cos\Big(\measuredangle \Big(v, \frac{p}{\|p\|}\Big)\Big).
\end{equation*}
\end{ex}

\begin{figure}[!htbp]
\includegraphics[width=\linewidth]{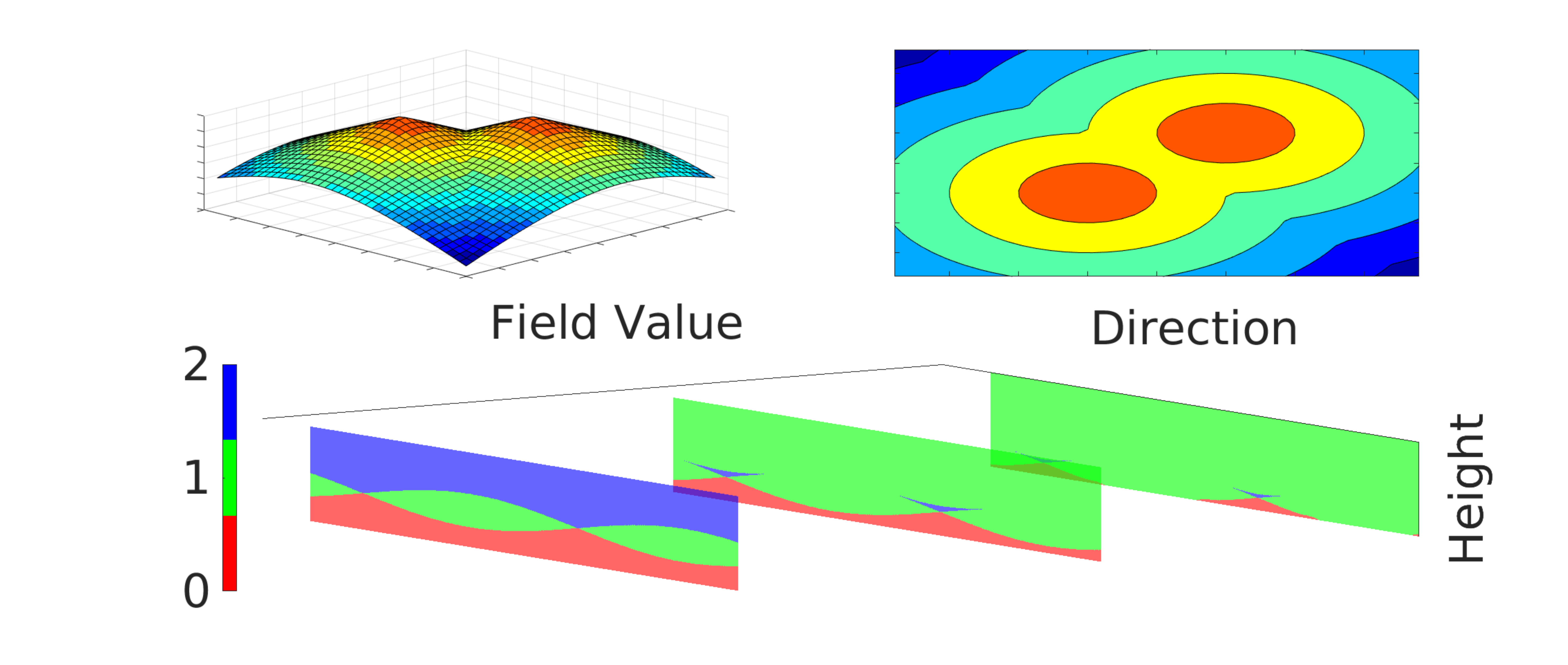}
  \caption{The Super Lifted Euler Characteristic Transform of the Euclidean distance from a point cloud consisting of 2 points, as discussed in Example \ref{exs}. Top Left: Function $f$. Top Right: The associated superlevel sets. Bottom: The lifted transform evaluated at the thresholds, with the stratification on $S^1 \times \mathbb{R} \times \mathbb{R}$ visualized at 3 slices. In terms of the Boolean algebra of the nerve $(S_1,S_2,S_1 \cap S_2)$, the red consitutes a single top stratum $(0,0,0)$, the green 3 top strata $(1,0,0), (0,1,0)$ and $(1,1,1)$, and the blue the top stratum $(1,1,0)$.}\label{Dparabola2}
\end{figure}

\subsection{Marginal distribution of the Super Lifts}
\label{sec:mar}

The Super Lifted Euler Characteristic Transform requires us to keep track of an extra parameter, $t$, in the transform space. For certain purposes, this parameter may be considered a nuisance parameter. In this subsection we show we can get rid of this parameter by integrating it out. This leads to marginal Euler curves. These marginal Euler curves are a continuous extension of the weighted Euler curves in that they agree with the weighted version whenever our data is discrete, but they also give the correct answer when the data is continuous. However, this comes with the price of losing the stratified map structure as the transform is no longer piecewise constant. The marginal Euler curves are defined as follows:

\begin{defn}
Let $f$ be a field $\mathbb{R}^d \rightarrow \mathbb{R}$. The marginal Euler curve $M_\nu^f$ of $f$ in direction $\nu$ is given by
$$
M_{\nu}^{f}(h)=\int_{\mathbb{R}} \SELECT(f)(\nu,h,t) \ dt.
$$

\end{defn}

This marginalization procedure gives us a continuous extension of the weighted Euler curve transform \cite{WECT}. Recall a weighted Euler curve was defined in \cite{WECT} as:
\begin{defn}
Let $K$ be a simplicial complex and $f: K \rightarrow \N$ an admissible weight function: constant on open simplices and $g(\tau)=\textrm{max}(g(\sigma \mid \tau < \sigma))$. The weighted Euler curve of $(K,g)$ in direction $\nu$ is

$$
\chi_{\nu}^{w}(r) = \chi^{w}( K \cap \{ x \cdot \nu \le r, g \},
$$

where

$$
\chi^{w}(K,g)= \sum_{d=0}^{\dim (K)} (-1)^{d} \sum_{\sigma \in K, \mid \sigma \mid=d} g(\sigma).
$$

\end{defn}

\begin{prop}
\label{prop:weighted}
Let $(K,g)$ be a weighted simplicial complex. The weighted Euler curves of $g$ coincide with the marginal Euler curves obtained from the Super Lifted ECT.
\end{prop}

\begin{proof}
At any direction height pair $(v,h)$ the sublevel set restricted to $(v,h)$ is definably homeomorphic to the lower star at $(v,h)$, which is a weighted simplicial complex. As $f$ is constant on open simplices, it suffices to show that the weighted Euler characteristic agrees with the marginalized Euler characteristic for any subcomplex of $K$. As $K$ is finite, function $f$ takes finitely many values $a_i$. Let $0 < a_1 < \ldots < a_n$.

Now observe that for $a_{i-1} < t \le a_{i}$, we have
$$
\chi(f(x) \ge t) = \chi \big( \bigcup_{k=i}^{n} (f(x)=a_k) \big) = \sum_{k=i}^{n} \chi(f(x)=a_k).
$$

We get a telescoping series

\begin{align*}
\int_{\mathbb{R}} \chi(f(x) \ge t) \ dt &  = \sum_{i=2}^{n} \int_{a_{i-1}}^{a_i} \chi(f(x)\ge t) \ dt + \int_{0}^{a_i} \chi(f(x) \ge t) \ dt \\
& = \sum_{i=2}^{n} \Big( \big(a_i-a_{i-1} \big) \sum_{k=i}^{n} \chi(f(x)=a_k) \Big) + a_i \sum_{i=1}^{n} \chi(f(x) = a_i) \\
& = \sum_{i=1}^{n} a_i \chi(f(x)=a_i).
\end{align*}

which agrees with the weighted Euler characteristic from \cite{WECT}.
\end{proof}

%%%%%%%%%%%%%%%%%%%%%%%
%% Equivariant Analysis
%%%%%%%%%%%%%%%%%%%%%%%

\subsection{Aligning fields with the lifted transforms}

A standard assumption in geometric morphometrics is that a shape does not change if one rotates or translates the shape. An initial step in almost all shape analysis is to register the shapes, that is, place them in the same position. The precise mathematical definition of the same position is not always agreed upon. The most common approach to registering shapes is to align landmarks---points on the shape that share a correspondence with points on every other shape. The problem of registration reduces to finding the rotations/reflections that best aligns the landmarks on all the shapes. In general, aligning or registering shapes is a challenging problem, see \cite{Puente} for an effective correspondence based approach.

One of the advantages of the ECT framework is that we require no correspondences between the shapes. A theoretical result in \cite{CMT} suggested that one may be able to use the ECT approach to align shapes without computing correspondences. The relevant result is, given the Lesbesgue measure $\mu$ over $S^{d-1}$, then the pushforward measure $\ECT_{*}(\mu)$ is invariant of $O(d)$ and suggests one can try and learn rotations and reflections that optimally align Euler curves. In \cite{Wai, SINATRA} algorithms were designed to align shapes in geometric morphometrics and proteins from molecular dynamic simulations, respectively.

In this subsection we describe a framework for comparing fields with the lifted transform without prior registration. This is achieved by the equivariance property: A rotation (or reflection) of the transform of a field is the same as the transform of a rotated/reflected field. This means that we can register the transforms instead of registering the fields. For shapes, this equivariance property of the ECT has been used in \cite{SINATRA} to compare shapes without prior registration. Using the lifted transforms, the same construction is possible for fields as the lifted transforms are
also $O(d)$-equivariant. Expressed mathematically, the equivariance property means that for $\phi$ an $O(d)$-action, the following diagram commutes:

\[
\begin{tikzcd}
\Def(\mathbb{R}^d) \arrow[r, "\phi"] \arrow[d, "\SELECT"]
& \Def(\mathbb{R}^d) \arrow[d, "\SELECT" ] \\
\CF(S^{d-1} \times \mathbb{R} \times \mathbb{R}) \arrow[r, "\phi"]
& \CF(S^{d-1} \times \mathbb{R} \times \mathbb{R})
\end{tikzcd}
\]

An $O(d)$-action $\phi$ operates on $\mathbb{R}^d$, so the dual action $\phi_{*}$ of $\phi$ on $\Def(\mathbb{R}^d)$ is defined through its inverse $\phi^{-1}$ via

\[
\begin{tikzcd}
\mathbb{R}^d \arrow[r, "f"] \arrow[d, "\phi"] & \mathbb{R} \\ 
\mathbb{R}^{d} \arrow[ru, "\phi_{*}f"']
\end{tikzcd}
\]

in other words

\begin{equation*}
\phi_{*} (f)(x)=f(\phi^{-1}(x)).
\end{equation*}

This is a composition of two definable functions and hence definable.

Similarly on $\CF(S^{d-1} \times \mathbb{R} \times \mathbb{R})$, the group $O(d)$ acts on $S^{d-1}$, and the action on constructible functions is given by

\begin{equation*}
\phi_{*}(f)(v,h,t)=f(\phi^{-1}v,h,t).
\end{equation*}

\begin{prop}
The  Lifted and the Super Lifted Euler Characteristic transforms are $O(d)$-equivariant.
\end{prop}

\begin{proof}
Let $f \in \Def(\mathbb{R}^d)$ and $\phi \in O(d)$. Now
\begin{align*}
\SELECT (\phi_{*} f)(v,h,t) & = \chi \big( \{ x \in \mathbb{R}^d \ \mid \ \langle x, v \rangle \le h, f(\phi^{-1}x) \ge t \} \big) & \mid\mid \ y=\phi^{-1}x \\
& = \chi \big( \{ y \in \mathbb{R}^d \ \mid \ \langle \phi y, v \rangle \le h, f(y) \ge t \} \big) \\
& = \chi \big( \{ y \in \mathbb{R}^d \ \mid \ \langle y, \phi^{-1}v \rangle \le h, f(y) \ge t \} \big) \\
& = \SELECT (f)(\phi^{-1}v,h,t) \\
& = \phi_{*} \SELECT (f)(v,h,t).
\end{align*}
The same reasoning applies to the lifted transform.
\end{proof}

%%%%%%%%%%%%%%%%%%%%%%%%%%%%
%% Equivariant Analysis ends
%%%%%%%%%%%%%%%%%%%%%%%%%%%%

%%%%%%%%%%%%%%%%%%%%%%%%%%%%
%A short demonstration of the equivariance principle
%%%%%%%%%%%%%%%%%%%%%%%%%%%%
\begin{ex}
\label{extraexample}
The equivariant analysis is straightforward to implement for two-dimensional fields. For $n$ scanning direction $V=\{v_1, \ldots, v_n \}$ picked uniformly on the sphere $S^1$, a cyclic permutation $p_j: [n] \rightarrow [n]: i \mapsto (i+j) \textrm{ mod}(n)$ induces an action on the Lie group $SO(2)$, the group of planar rotations, that sends direction $v$ to $v \cdot \big( \cos( \frac{2 \pi j}{n}), \sin( \frac{2 \pi j}{n}) \big)$. The directions $V$ form an orbit under the cyclic group, so that permutations of the scanning directions correspond to alignment up to the discretization error.

One can then align the fields by simply minimizing the SELECT distance as a function of cyclic permutations on the directions. Thanks to the stratification, the resulting distance function $\phi \mapsto \d_{SELECT}\big(F_1,\phi_{*} F_2\big)$ is differentiable almost everywhere. See Figure \ref{extra1} for an illustration.

\begin{figure}[!htbp]
\includegraphics[width=\linewidth]{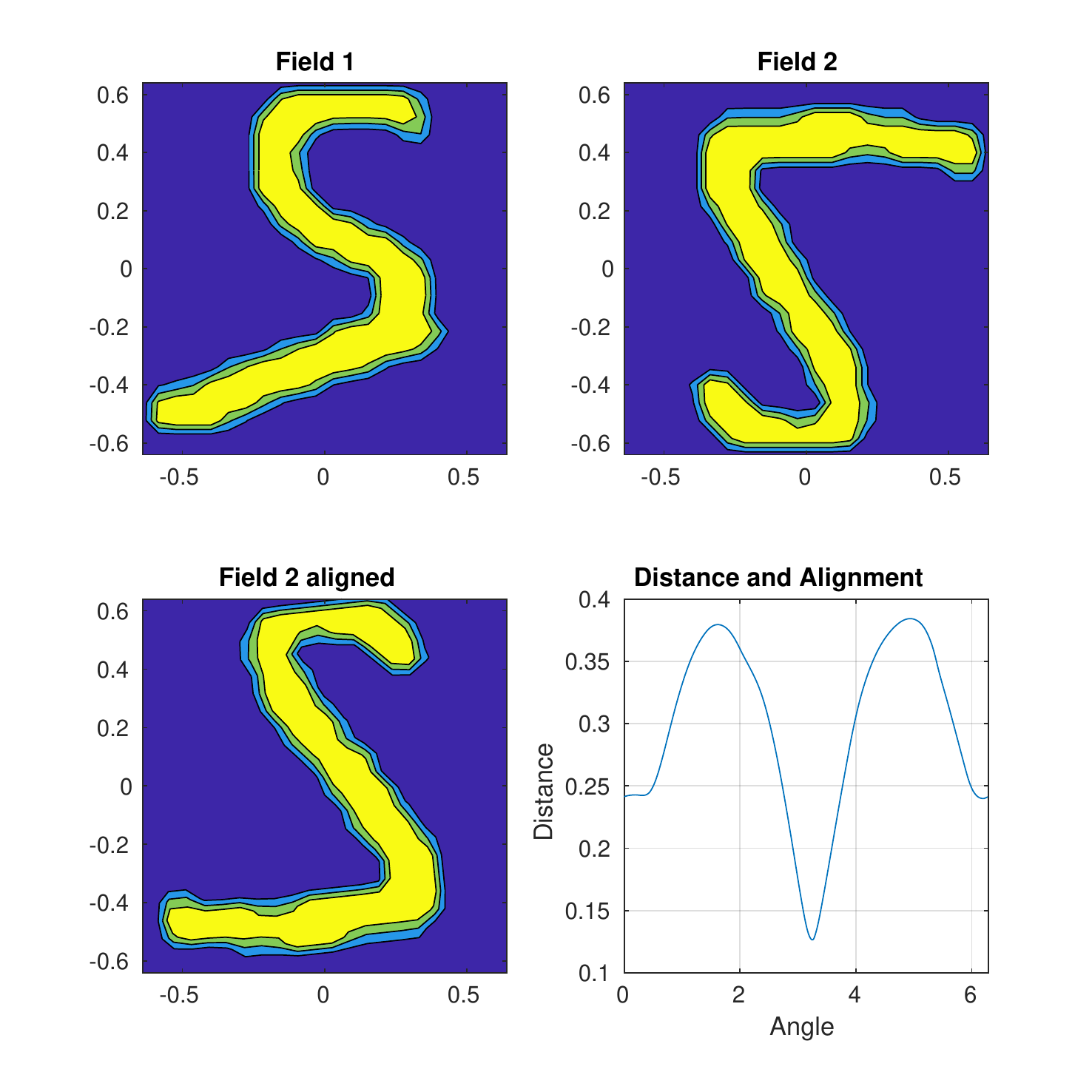}
\caption{A simple 2D-example of aligning 2 fields. Fields 1 and 2 both look like number 2, but Field 2 has the tail upside down. The alignment is obtained by minimizing the distance function with respect to the alignment angle.} \label{extra1}
\end{figure}
\end{ex}

%%%%%%%%%%
%SECTION 3
%%%%%%%%%%

\section{How many directions determine a field - a Moduli space for the lifted transform}
\label{sec:howmany}

In practice one cannot consider all directions in evaluating the Super Lifted Euler Characteristic Transform.
In this section we define classes of functions for which with a finite number of directions we can
distinguish any two functions in the class. This is interesting as the classes of functions we consider are infinite and uncountable
yet there is a finite-dimensional representation that can distinguish any two functions in the class.

For the Euler characteristic transform the proof of a finite bound on the number of directions that determine a shape  \cite{CMT} required restrictions on the shape space. In particular, the family of shapes considered were embedded geometric simplicial complexes 
$K\subset \mathbb{R}^d$. Embedded simplicial complexes have two very attractive properties: they are the standard digital representation
of a shape or surface and they have a great deal of geometric structure that allows for the quantification and verification of the
complexity of a shape.

Similarly, we also restrict the family of fields we consider to ensure sufficient regularity conditions such that we can provide an upper bound on the number of directions. In our case we will consider piecewise linear fields or functions. A piecewise linear function $f: \mathbb{R}^d \rightarrow \mathbb{R}$ is continuous, supported on a compact geometric simplicial complex, and the value of $f$ on faces of the simplicial complex 
is determined by linear interpolation of the vertices that span the face. Piecewise linear functions have two appealing properties.
First, piecewise linear functions are a very common digital representation of functions used in graphics and imaging. Second, piecewise linear
functions provide a simple and easy to verify set of criteria to determine the complexity of the field. The main result in this section
is that for a particular moduli space of piecewise linear functions we can upper bound the number of directions. 

The reason piecewise linear functions are a common digital representation is that one only needs to know the vertices and how they are connected to determine the field. Therefore, one can encode the function via a simplicial complex $K$ with the function values evaluated at the
vertices of the complex. Note that triangulations of a piecewise linear function are not unique and in this paper we consider different
triangulations of a piecewise linear function as equivalent or we consider fields the same if their graphs in $\mathbb{R}^d \times \mathbb{R}$ are the same.
We call the vertices of a triangulation the \emph{knots} of $f$ and the edges simply \emph{edges}. The moduli space of functions will
depend on properties of knots and edges, and defining the moduli space will require specifying what it means for an edge to be observable.

Our definition of an edge being observable will use the fact that the lifted transform can only change its value as it passes knots or edges of the field.

\begin{lem}
\label{lem:projection_is_linear}
Let $f: \mathbb{R}^d \rightarrow \mathbb{R}$ be a piecewise linear definable function, and $\nu$ a direction on $S^{d-1}$. For a fixed function value filtration $t$, the Super Lifted Euler curve of $f$ along $\nu$ can only change its value as $h$ passes a knot or an edge. The statement also holds for a fixed height $h$ as a function of $t$.
\end{lem}

This is a straightforward consequence of a result in \cite{BB-Morse} that says a Morse function on an affine cell complex can change its homotopy type only when passing through a vertex.

\begin{proof}
The map $x \mapsto (\nu \cdot x,f(x))$ is linear on a simplex since $f$ is linear. Any simplex then maps to a polygon whose boundary is given by the knots and edges of the simplex. Any vertical or horizontal halfplane cuts these polygons, creating a piecewise linear set whose vertices are either vertices of the original simplices, or intersections of edges of simplices with the halfplane. This is a piecewise linear set so Bestvina's result \cite{BB-Morse} (Lem 2.3) on homotopy to lower star sets applies.
\end{proof}

Lemma \ref{lem:projection_is_linear} provides criteria which allow for the definition of an edge to be observable. 

\begin{defn}
Let $e$ be a non-constant edge of a piecewise linear function $f$. We say edge $e$ is \define{observable} in direction $v$ and function value $t$ if the Euler curve of the superlevel set $f^{-1}\big([t, \infty)\big)$ changes at $x \cdot v, \ x \in e \cap f^{-1}(t)$.

We say the edge is \define{$\delta_k \times \delta_B$ observable} at $(v,t)$ if it is observable for any $(\nu, \tau) \in \big(B(v,\delta_k) \times B(t, \delta_B) \big)$.
\end{defn}

The reason we define observability both at a point as well as in a neighborhood is that we will require the stability induced by
having an observable neighborhood to define the moduli space of functions.

We are now ready to define a moduli space of functions:
\begin{defn}
\label{maindef3}
A function $f$ is in the \define{ class $\mathcal{K}(d,k,\delta_k,\delta_B)$ of functions} if it satisfies the following conditions:

\begin{enumerate}
\item \label{conPL} The function $f: \mathbb{R}^d \rightarrow \mathbb{R}$ is  continuous and supported on a compact geometric simplicial complex. The value of $f$ on the faces are a linear interpolation of the values at vertices that span the face. Without loss of generality, $f$ takes values in $[0,1]$.
\item \label{condeltaB} If $x$ and $y$ are two connected knots of $f$, the difference $|f(x)-f(y)|$ is at least $3 \delta_B$.
\item \label{condeltak} For any edge $e$, there exists a restriction $e_t := e \cap f^{-1}([t,\infty))$ and a ball of directions $B(v_0,\delta_k)$ such that $e_t$ is observable in any direction $v \in B(v_0,\delta_k)$.
\item \label{condeltall} For any pair $(v,t)$ of a direction and a function value there are at most $k$ jumps in any Euler curve restricted to $(v,t)$.
\end{enumerate}

\end{defn}

Condition (\ref{conPL}) states that the fields are linear and continuous. This is especially relevant for analyzing pixel data, such as images from MRI scans. The values at the knots are the pixel values, and the field is obtained by convolving the pixel signal with a rectangle size of the pixel. With this convolution we avoid the artificial pathologies introduced by working with cubical complexes, such as the failure of the Jordan curve theorem.

Condition (\ref{condeltaB}) ensures tameness of the inclusion maps in the function value filtration. This condition ensures that when we filter the field by decreasing function values, field behaves in a controlled fashion in the neighborhoods of its critical points. This can be seen as a Morse condition that ensures the critical points of $f$ are isolated and the function is not locally too flat around them.

Condition (\ref{condeltak}) provides regularity on the triangulation of the domain of the function relative to variation in function values, for example observing a neighborhood of an edge becomes more difficult for a skinny triangulation. Condition (\ref{condeltall}) again is a tameness condition that uniformly bounds changes in the Euler curve for direction and function value pair. This condition is used to control the combinatorial complexity of matching the critical points of the height function. The last three conditions together ensure that each stratum in the stratification of the sphere induced by the SELECT is sufficiently large.

We need some preliminary results and notation before we provide an upperbound for the above moduli space. Denote by $L_{t}^+$ the $t$-superlevel set $f^{-1}\big([t,\infty)\big)$. To keep track of the observability of edges, we need notion for a 
neighborhood of an edge of $f$. Here is a straightforward definition:
\begin{defn}
\label{def:nb1}
Let $e$ be an edge of $f$. An edge $e'$ is a \define{superlevel neighbor} of $e$ if there exists a superlevel set $L_{+}^t$ such that the vertices
$$
e \cap L_{+}^t, e' \cap L_{+}^t
$$
are connected by an edge.
\end{defn}

We will also make use of the following more operable definition:
\begin{defn}
\label{def:nb2}
A \define{combinatorial neighbor} of edge $e=(x_0,x_1) \in K$ is an edge $e' \in K$ if there is a triangle $(x_0,x_1,x_2) \in K$ such that any of the following holds:
\begin{enumerate}
\item[(I)] $f(x_0) < f(x_2) < f(x_1)$ and $e'=(x_0,x_2)$ or $e'=(x_2,x_1)$;
\item[(II)] $f(x_2) < f(x_0) < f(x_1)$ and $e'=(x_2,x_1)$;
\item[(III)] $f(x_0) < f(x_1) < f(x_2)$ and $e'=(x_0,x_2)$.
\end{enumerate}

We will call neighbors from cases (II) and (III) \define{dominating} and case (I) neighbors \define{dominated}. We extend this terminology to any vertices obtained by restricting the edges to superlevel sets.
\end{defn}

In other words, an edge dominates its neighbors in a triangle if it has the longest vertical length.
\begin{lem}
Definitions \ref{def:nb1} and \ref{def:nb2} of superlevel neighbors and combinatorial neighbors are equivalent.
\end{lem}

\begin{proof}
Clearly an edge $e'$ is a superlevel neighbor of $e=(x_0,x_1)$ only if $e'$ is an edge of a triangle containing $x_0$ and $x_1$. For $e'$ a dominating neighbor with endpoint $x_2$, any superlevel set that intersects $e$ also intersects the triangle $(x_0,x_1,x_2)$ at $e$ and $e'$. For a dominated neighbor, a superlevel set that intersects triangle $(x_0,x_1,x_2)$ always intersects $e$, and one of the dominated neighbors depending on whether or not $f(x_2)$ is above or below the superlevel set.  
\end{proof}

Next we show that in our moduli space observability of an edge is uniform in function value filtration.
\begin{lem}
\label{lem:unifobs}
Let $f$ be in class $\mathcal{K}(d,k,\delta_k,\delta_B)$. If edge $e$ of $f$ is $\delta_k$-observable in direction $v$ at some superlevel set $L_t^{+}(f)$, then $e$ is also $\delta_k$-observable in direction $v$ for any superlevel set $L_{t'}^{+}(f)$ whenever $e_{min} < t' < e_{max}$.
\end{lem}
\begin{proof}

The key idea is that an edge projects to its dominating neighbor linearly.

Let $e$ be an edge $(x_0,x_1)$ such that $f(x_1)>f(x_0)$. Write $p(t)$ for the intersection point of $e$ and a $t$-superlevel set of $f$, that is

$$
p(t) = e \cap L_{t}^{+}, \ f(x_0) < t < f(x_1).
$$ 

The observability of $p(t)$ is a local property: it depends on the link of $p(t)$, or the angles between $p(t)$ and its neighbors. It suffices to show that for any neighbor $q(t)$ of $p(t)$ the angle between $q(t),p(t)$ and $e$ is constant in $t$.

Let $e'$ be a dominating neighbor of $e$. This means $f$ attains either a higher or lower value somewhere on $e'$ than anywhere on $e$. If higher, let $\pi$ be the point on $e'$ where $f$ attains value $f(x_1)$. Then for any $t \in [0,1]$,

$$
p(t) = (1-t)(x_0,f(x_0))+t(x_1,f(x_1)),
$$

and

$$
q(t) = (1-t)(x_0,f(x_0))+t(\pi,f(\pi)).
$$

For any choice of $t$, the triangle $(x_0,p(t),q(t))$ is congruent to $(x_0,x_1,\pi)$ with scale $t$, in particular, the angles are constant. With the obvious changes, the same argument holds when $e'$ attains a lower value with congruence to the triangle $(x_1,p(t),q(t))$.

For $e'$ a dominated neighbor of $e$, let $\pi$ be the point on $e$ where $f$ attains value $f(x_2)$. Again a linear interpolation on $(x_0,x_2)$ projects to a linear interpolation on $(x_0,\pi)$, while an interpolation on dominated neighbor $(x_2,x_1)$ projects to an interpolation on $(\pi,x_1)$.
\end{proof}

This uniform lower bound on observability implies the following corollary:

\begin{cor}
\label{circ1}
For $f \in \mathcal{K}(d,k,\delta_k,\delta_B)$, every superlevel set of $f$ is in the class $\mathcal{K}(d,k,\delta_k)$.
\end{cor}

Our discretization result pertains to finitely many Euler scans. We make this explicit with the following definition:

\begin{defn}
Let $f: \mathbb{R}^d \rightarrow \mathbb{R}$ be a definable function. The \define{Euler scan} 
$S^f_{\nu,t}(h): \mathbb{R} \rightarrow \mathbb{R}$ of $f$ in direction $\nu$ and function value $t$ is defined as
$$
S^f_{\nu,t}(h)= \SELECT(f)(\nu,t,h).
$$
\end{defn}

We are now ready to state the main theorem of this section:

\begin{thm}
\label{mainthm}
Any function in $f \in \mathcal{K}(d,k,\delta_k,\delta_B)$ can be determined by SELECT using no more than
$$\Delta(d,\delta,k_\delta,\delta_B) = \left[ \left((d-1) \, k_\delta \,  +1 \right) \left(1+\frac{3}{\delta}\right)^d+O\left(\frac{d^{d+1}k_\delta^{2d}}{\delta^{2d(d-1)}}\right) \right] \times \left\lfloor \frac{1}{ \delta_B}\right\rfloor$$
Euler scans.
\end{thm}

\begin{proof}
Note that Theorem 7.14 of \cite{CMT} stated an upper bound on the number of directions needed to determine a shape using the ECT of 
$$
\Delta(d,\delta,k_\delta)= \left((d-1) \, k_\delta \,  +1 \right) \left(1+\frac{3}{\delta}\right)^d+O\left(\frac{d^{d+1}k_\delta^{2d}}{\delta^{2d(d-1)}}\right).
$$
We will use this result in our proof.

We  obtained our bound by taking the superlevel sets along thresholds $1-\delta_B,1- 2 \delta_B, \ldots,  1 - \lfloor \frac{1}{ \delta_B}\rfloor \delta_B$, and reconstructing each piece starting from the top. Then it suffices to show we can reconstruct $f$ by interpolating the superlevel sets at the given thresholds. We show this procedure is enough to recover the knots and edges of $f$, beyond that the function is linear. We start the reconstruction from top to bottom, so at any given superlevel set we may assume we know all the vertices in any superlevel set above the given superlevel set.

Function $f$ takes values in $[0,1]$, and each of its knots is vertically at least $3\delta_B$ apart from its neighbors. This implies that each of the superlevel sets $L_{1-\delta_B}^{+}, L_{1-2\delta_B}^{+}$ is either empty or a disjoint union of polygons around the apexes of $f$.

For $x_0$ an apex such that $1- \delta_B \le f(x_0) \le 1$, the connected component containing $x_0$ intersected with the superlevel sets $L_{1-\delta_B}^{+}, L_{1-2\delta_B}^{+}$ are congruent polytopes with scale $(f(x_0)-1+\delta_B):(f(x_0)-1+2\delta_B)$, and the intersections with superlevel sets in between are given by linear interpolation. This interpolation extends all the way to $t=1$ until the polytopes contract to the apex $x_0$. With these two superlevel sets we learn all the superlevel sets down until $1-2\delta_B$, the locations of the apexes and the function values, and the behavior of the edges in their neighborhoods.

By matching the face normals of these polytopes, we obtain a matching for the vertices. We then know that the neighbors of $x_0$ are located on the rays starting from $x_0$ that pass through the vertices of the polytopes. With the obvious changes, a similar reconstruction strategy works for any local maximum of $f$, which we identify from a new connected component in the superlevel set.

Other than local maxima, $f$ may have two other kinds of knots:
\begin{enumerate}
\item knots that have at least 2 neighbors with a higher function value;
\item knots that have exactly 1 neighbor with a higher function value.
\end{enumerate}

Let $x_0$ be a knot with at least 2 neighbors that have a higher function value, and let $U$ and $L$ be two consecutive superlevel sets such that $x_0$ is contained in $L$ but not $U$.

As $x_0$ has 2 neighbors above it, from $U$ we already anticipate $x_0$ at the intersection of the rays coming from the neighbors of $x_0$ that are in $U$. Call this anticipated position $\tilde x_0$. When we see $L$, and verify that indeed $\tilde x_0 \in L$ and that the line segments from $\tilde x_0$ to its neighbors above are also contained in $L$, we then know that $\tilde x_0 =x_0$, and also how $x_0$ connects to the neighbors above it. The other neighbors of $x_0$ are vertically at least $3\delta_B$ lower than $x_0$, and therefore the edges from $x_0$ to them are on the boundary of $L$ when restricted to the superlevel set. As we already know the knots and edge rays from the earlier superlevel sets, we simply connect $x_0$ to all the boundary edges of $L$ that we can without intersecting the known connections. In the case $f(x_0)$ is exactly the value corresponding to the superlevel set $L$, we need another superlevel set $L'$ below $L$ for reconstructing $x_0$ and its link. This is again guaranteed to be bounded away from lower neighbors of $x_0$, because they are at least $3\delta_B$ apart vertically.

Finally, there are knots that have exactly one neighbor with a higher function value. Again let $x_0$ be such a knot and $U$ and $L$ be two consecutive superlevel sets such that $x_0$ is contained in $L$ but not $U$. Let $x_1$ be the neighbor of $x_1$ located higher than $x_0$. Let us first assume $x_0$ has some other neighbor(s) as well.

The face normals associated with the edge $(x_1, x_0)$ have to change as we go below $x_0$ in the superlevel sets, if $x_0$ is vertex that does not merely subdivide an edge, in which case $x_0$ is irrelevant for reconstructing $f$. The other way face normals can change is when the superlevel sets go below another edge $(x_2,x_1)$ that is adjacent to $(x_1,x_0)$. However, as $x_0$ is at least $3 \delta_B$ higher than $x_2$, we can determine whether the superlevel sets went below edge $(x_0,x_1)$ or $(x_2,x_1)$ by going to the next superlevel set. In doing this we also learn the edges around $x_0$ going to its neighbors with lower function values. We can the pinpoint the exact location of $x_0$ by intersecting these lines.

In the case $x_0$ has exactly 1 neighbor $x_1$, by continuity of $f$, $f(x_0)=0$. As this neighbor is above them, and we can find $x_0$ by interpolating the rays from the vertices discovered earlier. The exact cut point is given by the fact that $f(x_0)=0$.
\end{proof}

\subsection{Practical Considerations}

Theorem \ref{mainthm} provides a theoretical guarantee on how many many scans are needed to distinguish certain fields. It does not, however, specify how fine of a discretization is needed in practice for a good representation. In this subsection, we address this issue, focusing on the field value filtration.

When discretizing the transform, an important consideration is that our goal is to numerically evaluate the integral $\d_{SELECT}(f_1,f_2)$. For shapes, it is a standard practice to pick the directions and heights at uniform intervals and simply take the mean over the matrices. However, with fields we have an extra parameter that scales the computation time linearly. Oftentimes a better idea is to select non-uniform thresholds, and then adjust the weights accordingly with e.g. the trapezoid rule.

\begin{ex}
\label{parameterexample}
Continuing the 2D example from Example \ref{extraexample}, we have 2 piecewise linear fields. Both of these fields are saturated at $0$ and $1$, meaning that they do not exactly adhere to the Condition \ref{condeltaB} of Definition \ref{maindef3}. However, the superlevelsets of these plateaus are still valid shapes and we know how many directions we need to tell them apart.

Elsewhere, the minimum of the gradient for either of the fields is $3.9 \times 10^{-3}$, so the $\delta_B$ for both of the fields is $1.3 \times 10^{-3}$, meaning that theoretically, some 770 field value filtrations would suffice. However, in practice, a smaller representation is likely to be good enough.

In Figure \ref{extra2} we see the evolution of the distances of the fields as we vary the discretization parameters. The transformations are computed on a discretized grid of size $(720,300,100)$, representing the numbers of directions, height filtrations and field value filtrations, respectively. We then subsample each of the grid dimensions pruning them uniformly and study the resulting approximation for the integral.

\begin{figure}[!htbp]
\includegraphics[width=\linewidth]{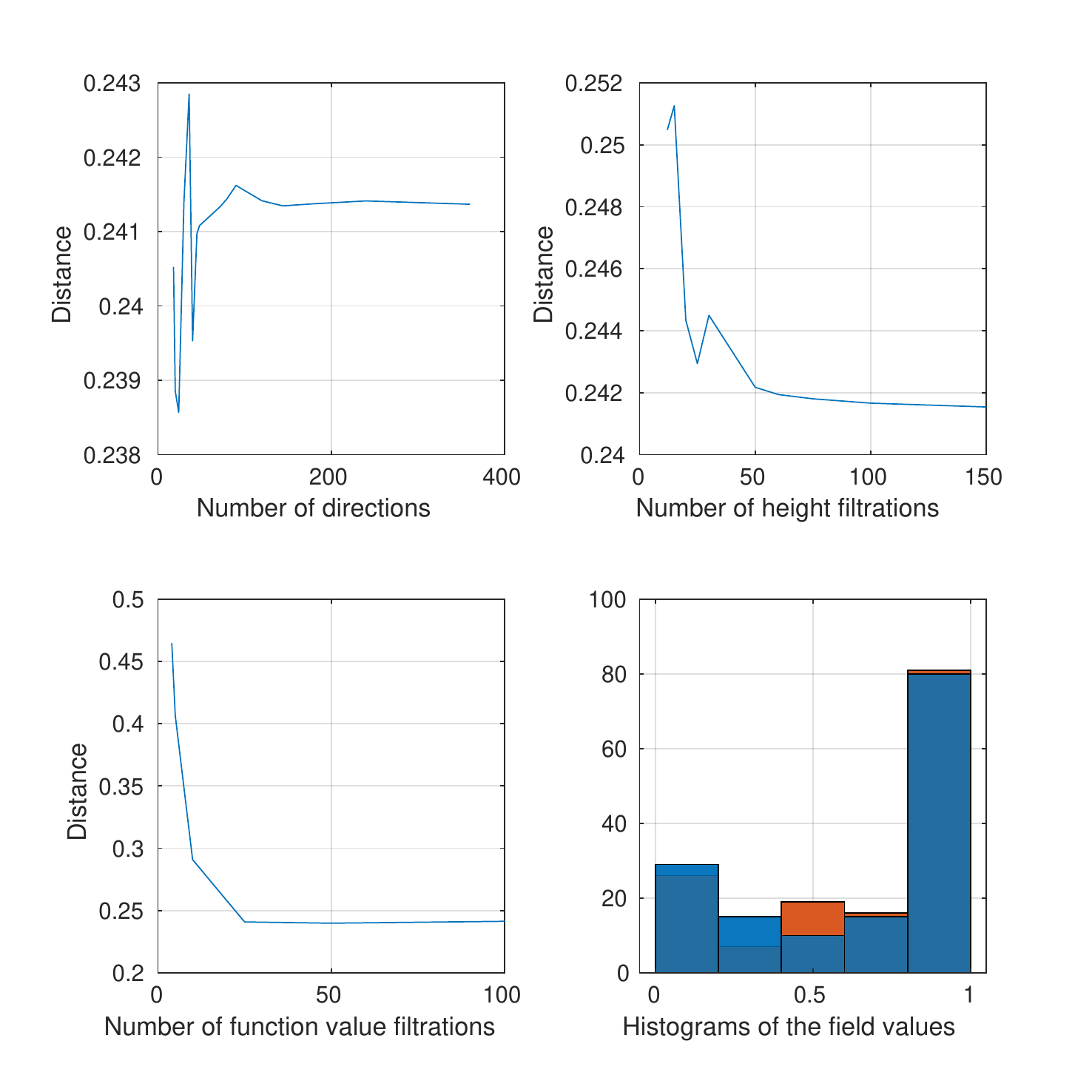}
\caption{ Example \ref{parameterexample}. Distance of the two fields as we vary the discretization parameters. The data concentrates towards higher field values, indicating there is likely more information there.} \label{extra2}
\end{figure}

We see that the discretization is stable starting from grid size $(150,50,25)$. However, by looking at the histogram of the field values (Figure \ref{extra2}), we see that the uniform sampling is likely inefficient. We can obtain a reasonable representation of the distance by concentrating more on field value thresholds at around $0.8-1$. For instance, in this example, just the thresholds at function values $(0.4,0.8,0.85,0.95,1)$ give normalized distance $0.2574$ (cf. $0.2414$ with 100 filtrations), which accounts for a 4-fold improvement in computation speed over the 20 filtrations (where $d=0.2575$).

\end{ex}

%%%%%%%%%%%%%%%%%
%END OF SECTION 3
%%%%%%%%%%%%%%%%%

\section{Applications to simulated and real data}
\label{sec:applications}
In this section we provide a proof-of-concept application of the Super Lifted Euler Characteristics Transform's utility in modeling fields.
First, we consider a clustering problem on a simulated dataset. Then, we use SELECT for predicting various molecular markers and subtypes based on MRI images in a breast cancer dataset. We couple the transform with a standard classifier model and compare the results against state-of-the-art machine learning methods.

The results on the simulated data provide an intuition on how the transform can be used to compare fields and which properties of the fields the transform keys in on. For the cancer application, we find that using an off-the-shelf classifier model with a SELECT representation provides comparable classification accuracy to a highly-optimized deep neural network. The state-of-the-art machine learning methods use a specifically tailored model for each classification task, and each of these models is a result of extensive fine tuning of parameters and inputs based on vast research on feature extraction from the fields. Our classifier is an off-the-shelf classifier with standard parameters. In fact, we use the same model for all tasks. 

\subsection{Simulations}

We look at the differences of the Super Lifted Euler Characteristic Transform between four qualitatively different families of fields.
The fields are evaluated on a $3$-dimensional grid consisting of $1,000$ points in the cube $[-1,1]^3$ where
$$
(x_i,y_j,z_k)= \left(-1 +\frac{2i}{9},-1 +\frac{2j}{9},-1 +\frac{2k}{9}\right), \quad i,j,k = 0,1,\ldots,9.
$$

\begin{figure*}[!htbp]
    \centering
    \begin{subfigure}[b]{0.225\textwidth}
        \centering
        \includegraphics[width=\textwidth]{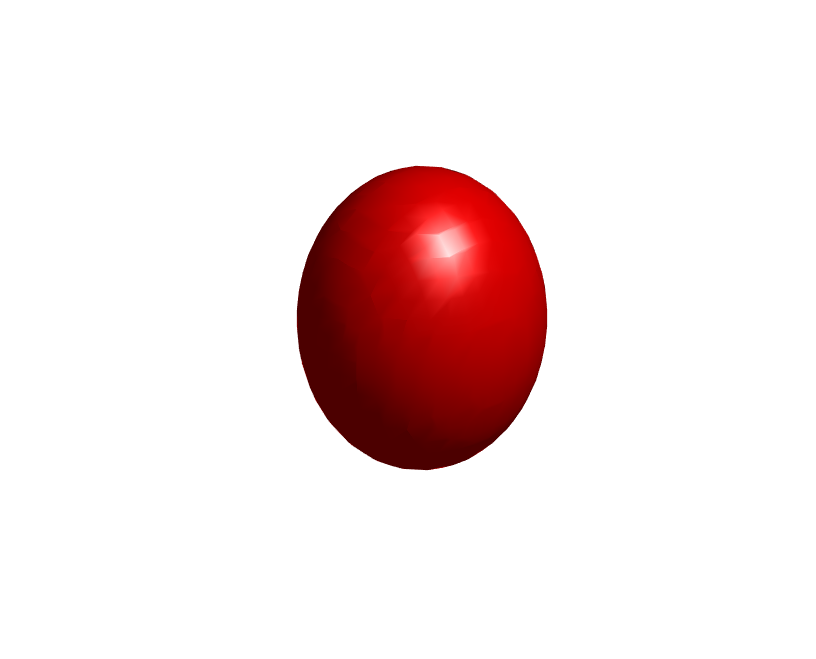}
        {{\small Group 1}}    
    \end{subfigure}
    \begin{subfigure}[b]{0.225\textwidth}  
        \centering 
        \includegraphics[width=\textwidth]{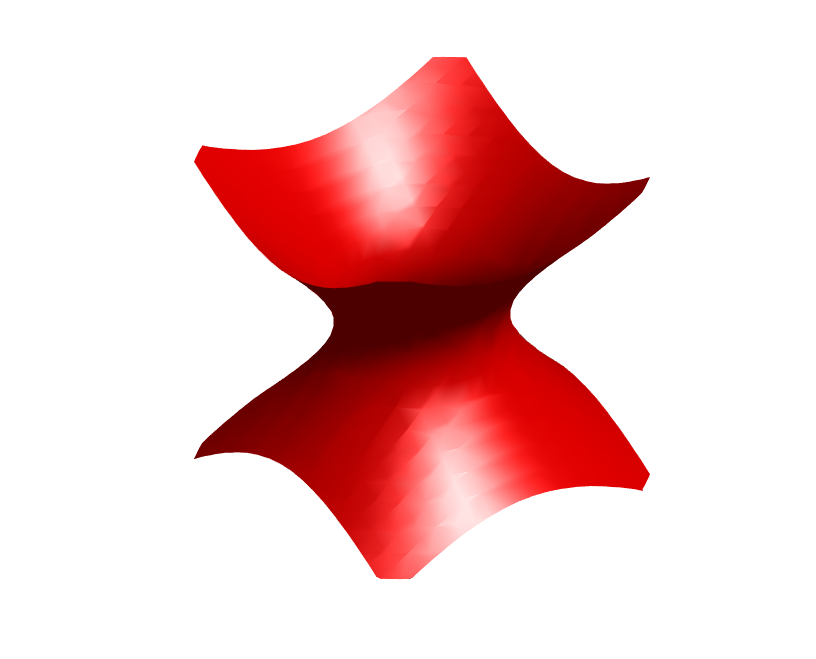}
        {{\small Group 2}}    
    \end{subfigure}
    \begin{subfigure}[b]{0.225\textwidth}
        \centering
        \includegraphics[width=\textwidth]{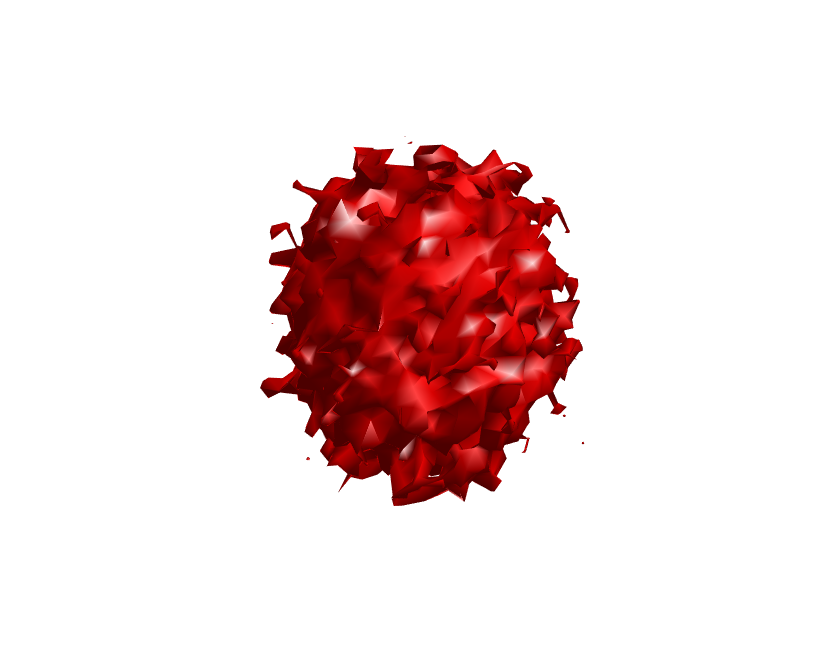}
        {{\small Group 1}}    
    \end{subfigure}
    \begin{subfigure}[b]{0.225\textwidth}  
        \centering 
        \includegraphics[width=\textwidth]{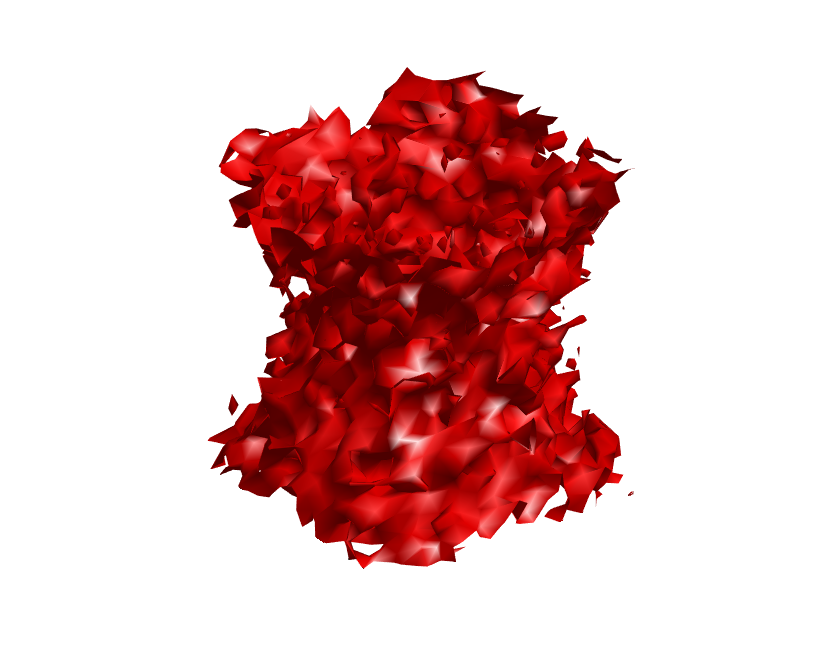}
        {{\small Group 2}}    
    \end{subfigure}
    \vskip\baselineskip
        \begin{subfigure}[b]{0.225\textwidth}   
        \centering 
        \includegraphics[width=\textwidth]{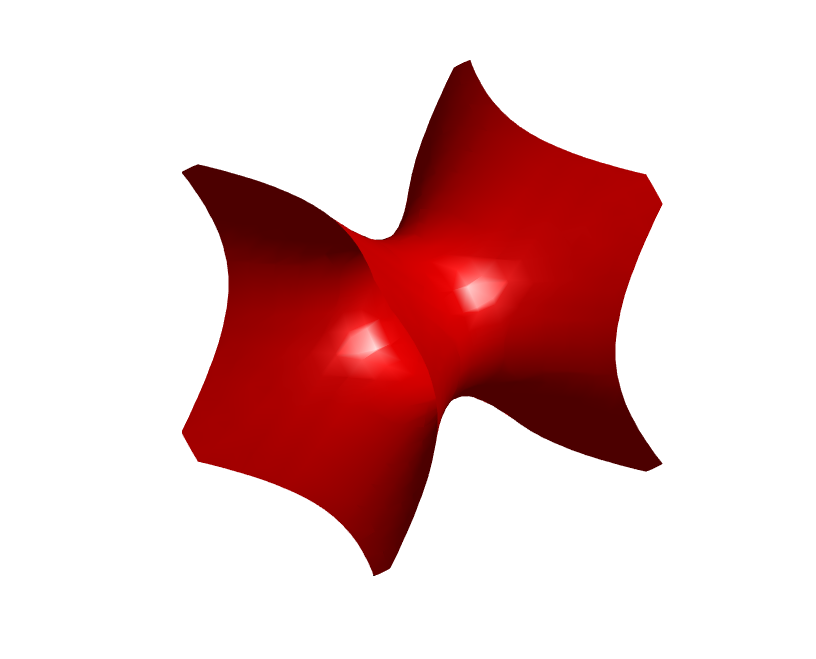}
        {{\small Group 3}}    
    \end{subfigure}
    \begin{subfigure}[b]{0.225\textwidth}   
        \centering 
        \includegraphics[width=\textwidth]{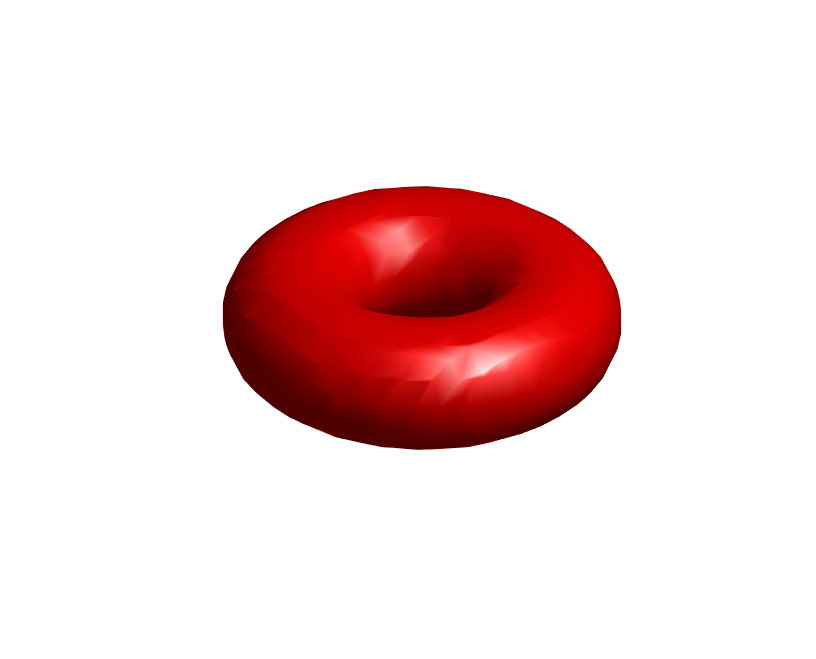}
        {{\small Group 4}}    
    \end{subfigure}
    \begin{subfigure}[b]{0.225\textwidth}   
        \centering 
        \includegraphics[width=\textwidth]{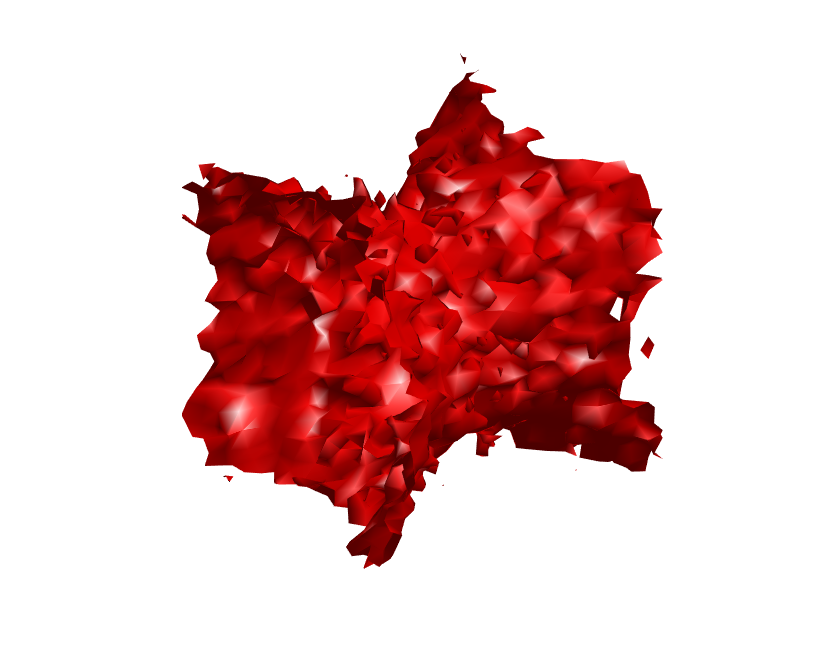}
        {{\small Group 3}}    
    \end{subfigure}
    \begin{subfigure}[b]{0.225\textwidth}   
        \centering 
        \includegraphics[width=\textwidth]{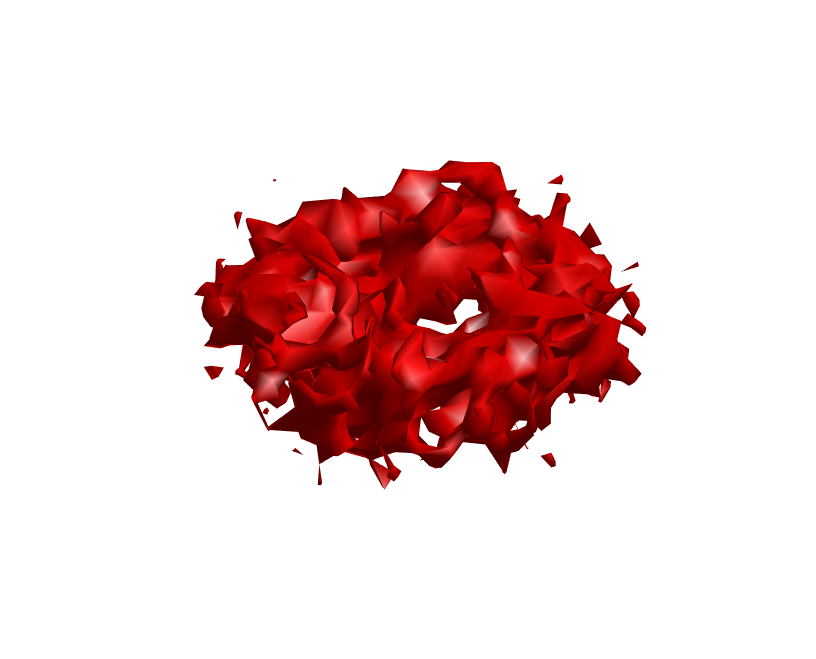}
        {{\small Group 4}}    
    \end{subfigure}
      \caption{Example contours of the fields. On the left the simulation setup 1 with uncorrupted fields, and on the right setup 2 with the added noise.}
\label{fig:figs2}
\end{figure*}

Each of the family of fields have three parameters  $\alpha, \beta, \gamma >0$. We generate $10$ fields for each of
the $4$ families, so for field $F_{l,t}$ the family index is $t$ and $l$ indexes the random field from family $t$. The following
simulation procedure was used to generate the random fields
\begin{align*}
t=1: F_{l,t}(x_i,y_j,z_k) & = \alpha_{l,1} x_i^2+ \beta_{l,1} y_j^2+ \gamma_{l,1} z_k^2 + \epsilon_{1,l,i,j,k}; \\
t=2: F_{l,t}(x_i,y_j,z_k) & = \alpha_{l,2} x_i^2+ \beta_{l,2} y_j^2- \gamma_{l,2} z_k^2 + \epsilon_{2,l,i,j,k}; \\
t=3: F_{l,t}(x_i,y_j,z_k) & = \alpha_{l,3} x_i^2- \beta_{l,3} y_j^2- \gamma_{l,3} z_k^2 + \epsilon_{3,l,i,j,k} ; \\
t=4: F_{l,t}(x_i,y_j,z_k) & = \big(\sqrt{\alpha_{l,4} x_i^2+ \beta_{l,4} y_j^2} - \delta_{l,t}\big)^2 + \gamma_{l,4} z_k^2 + \epsilon_{4,l,i,j,k},
\end{align*}
where $\alpha_{l,t}, \beta_{l.t}, \gamma_{l,t} \stackrel{iid}{\sim} \mbox{U}[.5,1]$, and independently for $t=4$ $\delta_{l,t} \stackrel{iid}{\sim} \mbox{U}[.4,.6]$.

\begin{figure}[!thbp]
\begin{subfigure}{0.48\textwidth}
\includegraphics[width=\linewidth]{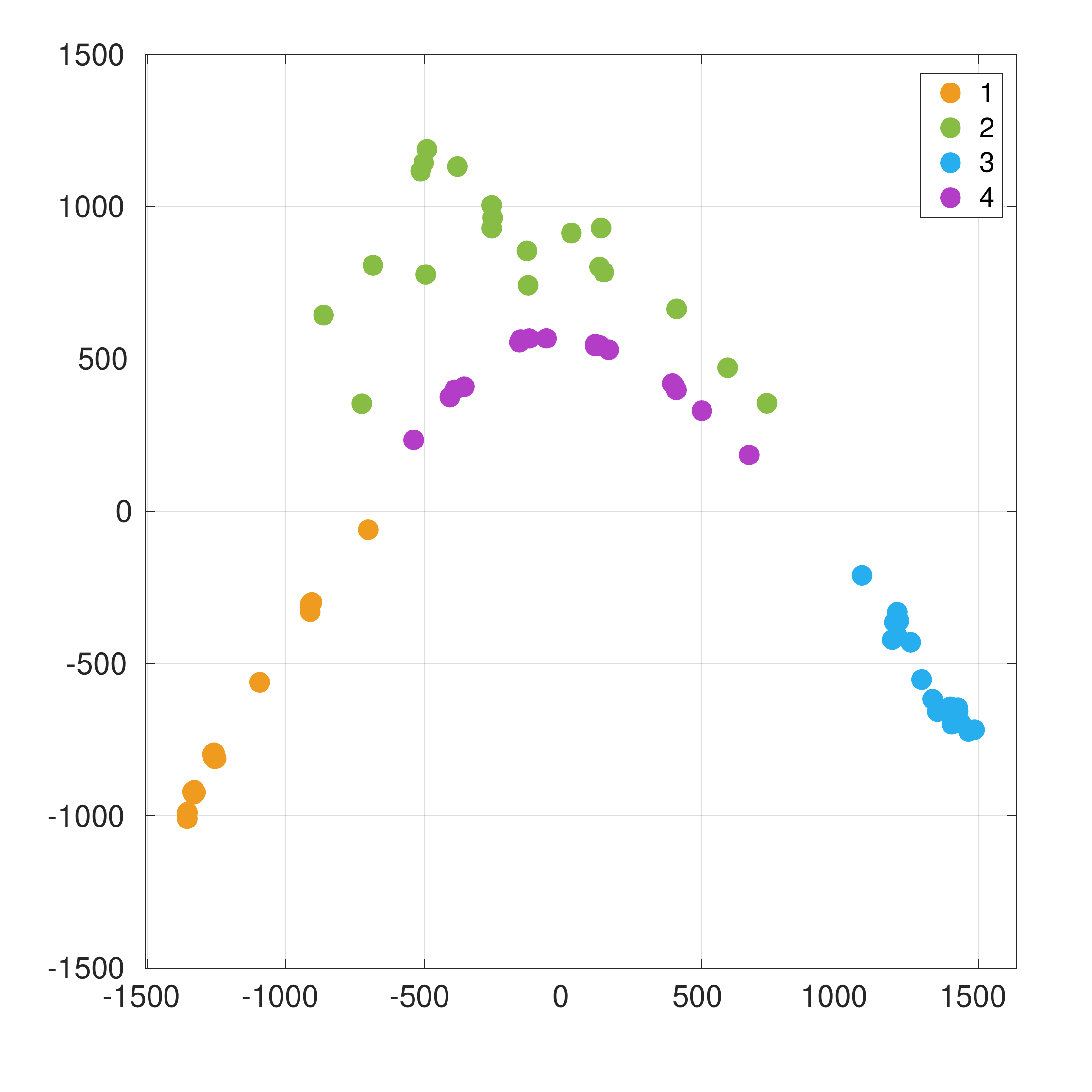}
\caption{MDS plot}
\end{subfigure}
\begin{subfigure}{0.48\textwidth}
\includegraphics[width=\linewidth]{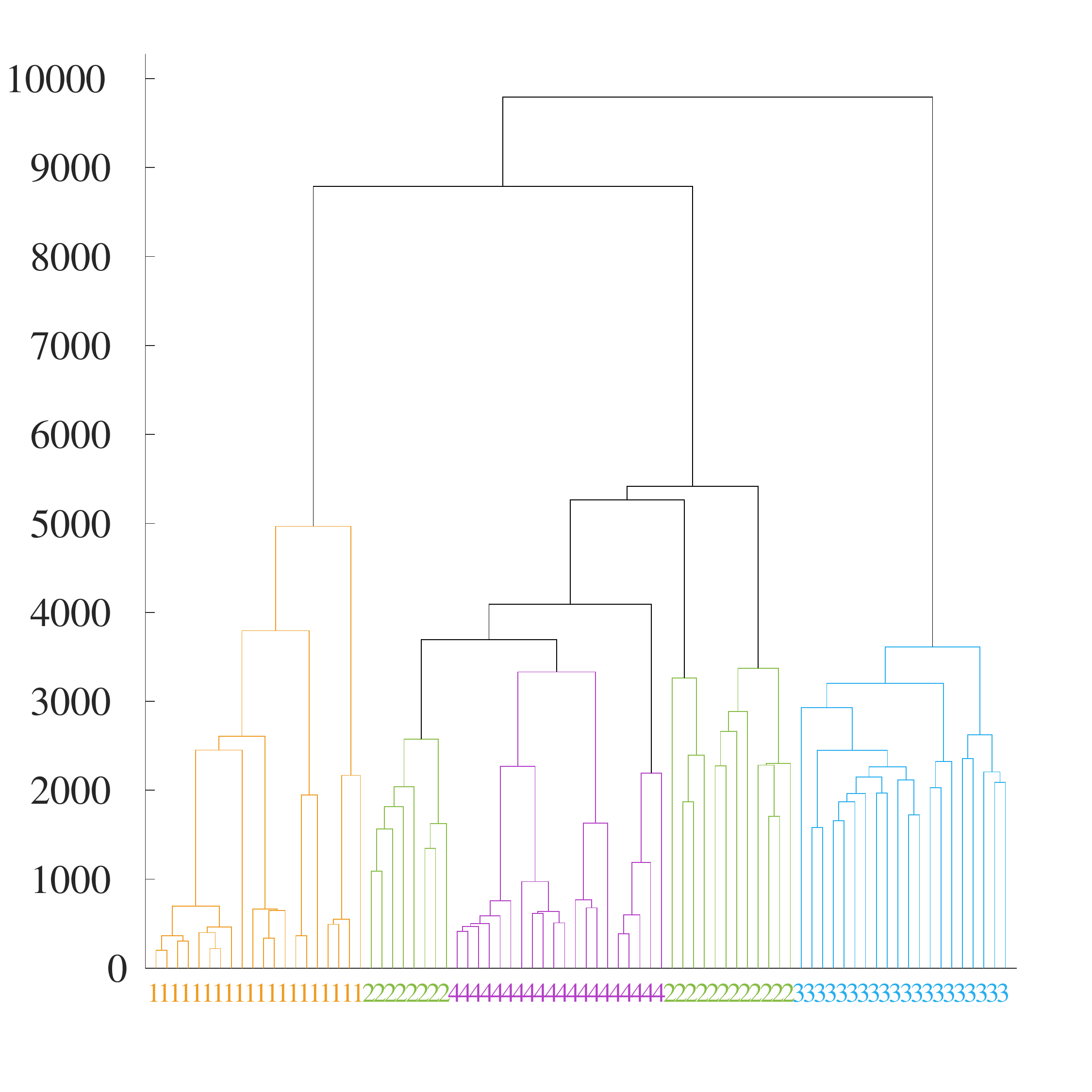}
\caption{Dendogram}
\end{subfigure}
\medskip
\begin{subfigure}{0.48\textwidth}
\includegraphics[width=\linewidth]{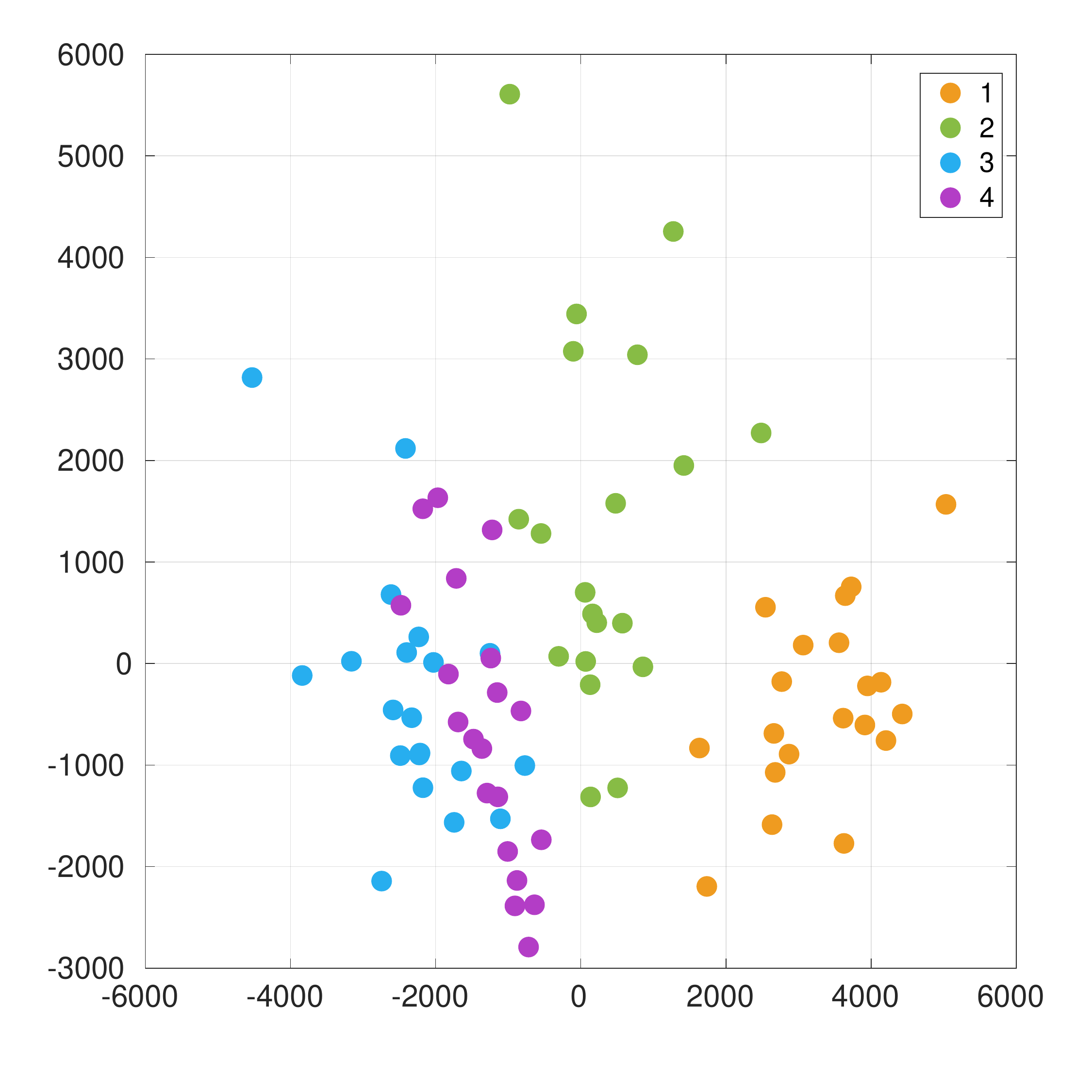}
\caption{MDS plot}
\end{subfigure}
\begin{subfigure}{0.48\textwidth}
\includegraphics[width=\linewidth]{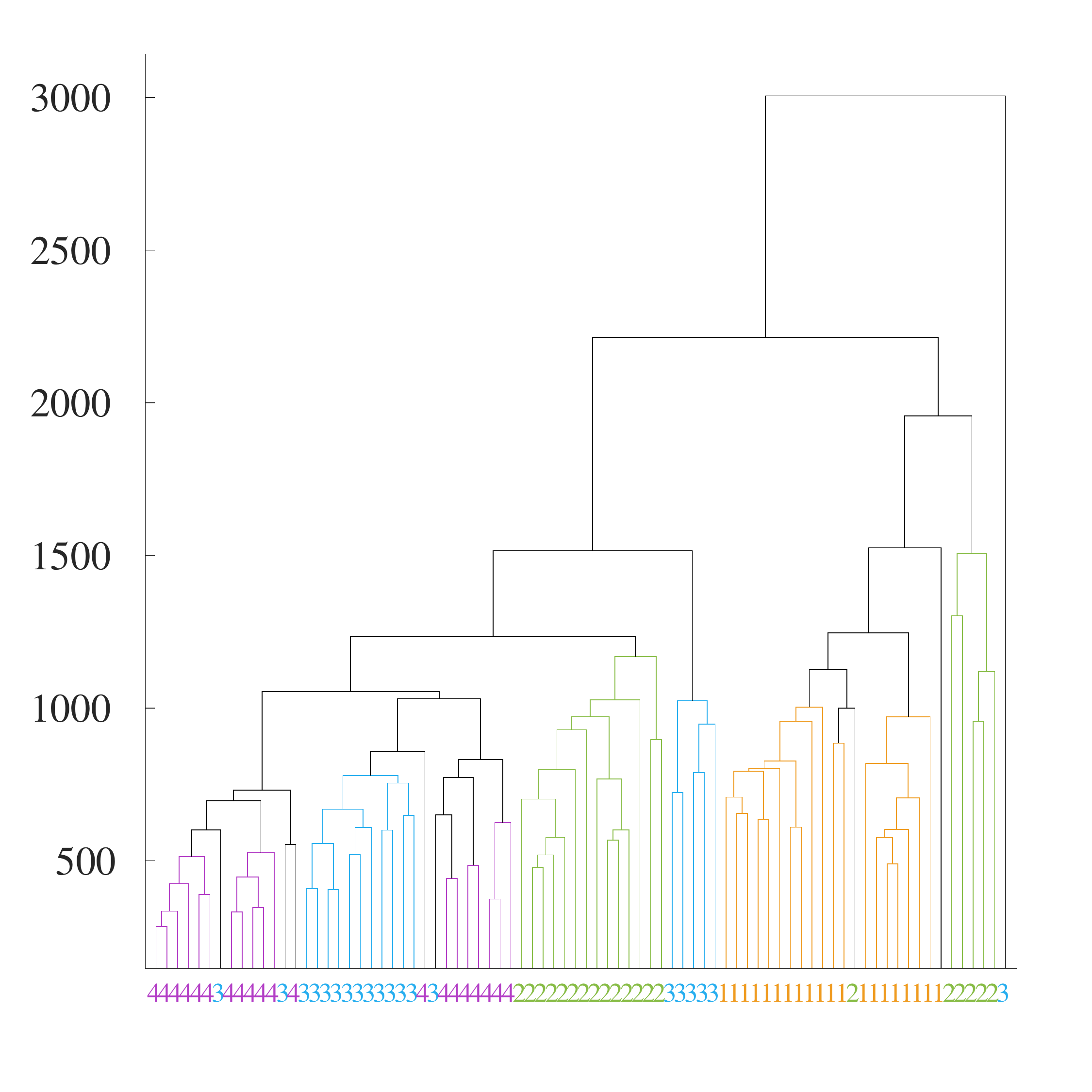}
\caption{Dendogram}
\end{subfigure}
\caption{Graphical summaries of the simulations findings. Figures (a) and (c) are Multidimensional Scale plots (MDS) and (b) and (d) are dendograms. Figures (a)-(b) are for the noiseless setup and Figures (c)-(d) are the setup with noise.} \label{fig:figs3}
\end{figure}

In terms of the noise model we consider two simulation setups:
\begin{enumerate}
\item[Setup 1:] $\epsilon_{t,l,i,j,k}=0$ ;
\item[Setup 2:] $\epsilon_{t,l,i,j,k} \sim^{\textrm{iid}} \mathcal{N}(0,.1)$, independently of the coefficients $\alpha, \beta, \gamma$.
\end{enumerate}

In Setup 1, the fields are noiseless quadratics with varying coefficients determined by parameters $\alpha, \beta$ and $\gamma$. Geometrically, these parameters control  their scale and eccentricity. In the second setup each field is corrupted by Gaussian noise. 
Contours of a representative from each family in both setups are displayed in Figure \ref{fig:figs2}.

We scale the fields by a global constant such that all of them take values in $[0,1]$. Then we pick 30 thresholds uniformly on $[0,1]$, and compute the corresponding superlevel sets by matlab isosurface procedure. To compute the transform, we choose 362 roughly uniformly spread directions over $S^2$, and evaluate the Euler curves at 100 height value filtrations for each superlevel set. Each field is then represented by an array of size $362 \times 100 \times 30$.

Graphical summaries of the clusters formed by the SELECT distances between the fields and families are presented in Figure \ref{fig:figs3}. In the noiseless setup, the dendogram in Figure \ref{fig:figs3} (b) shows a clear clustering of the families with groups 2 and 4 being closest. These two groups are similar in that both families take a cylindrical shape in the $xy$-plane, as is depicted in Figure \ref{fig:figs2}. Group 1 is distinct from the rest in that it grows in all directions: Group 3 only grows in 1 direction and groups 2 and 4 are somewhere in between. Groups 2 and 4 also share the same axis of symmetry.

When we corrupt the signal the groups get closer to each other, as is evident from Figure \ref{fig:figs3} (c). We see that group 1 is quite well separated from the rest, and that group 4 overlaps with groups 2 and 3. These groups are different from Group 1 in that they have regions of high curvature. When noise gets added to the center of the curvature, this has the potential of creating loops in the superlevel sets, thus making these fields similar in their overall complexity. On the other hand the curvature in group 1 is almost the same everywhere, so the added noise may create a new connected component but it is unlikely to create a loop.

\subsection{Application to Breast Cancer Imaging Dataset}

To assess the applied utility of the method on a real life dataset, we test how well the Super Lifted Euler Characteristic Transform can predict different types of molecular subtypes and markers based on MRI sequences. We use a freely available Breast Cancer Dataset\footnote{\url{https://sites.duke.edu/mazurowski/resources/breast-cancer-mri-dataset/}} that was analyzed in \cite{Saha}. For benchmarking, we compare the method against the imaging based machine learning methods studied in \cite{Saha}. These benchmark models build on imaging features engineered through extensive medical research based and they represent the performance achievable by current understanding of what explicit feature extraction can achieve. We consider seven binary classification tasks presented in that paper. The details are in Appendix \ref{secA2}.

\begin{figure}[h]
\centering
\includegraphics[width=12cm]{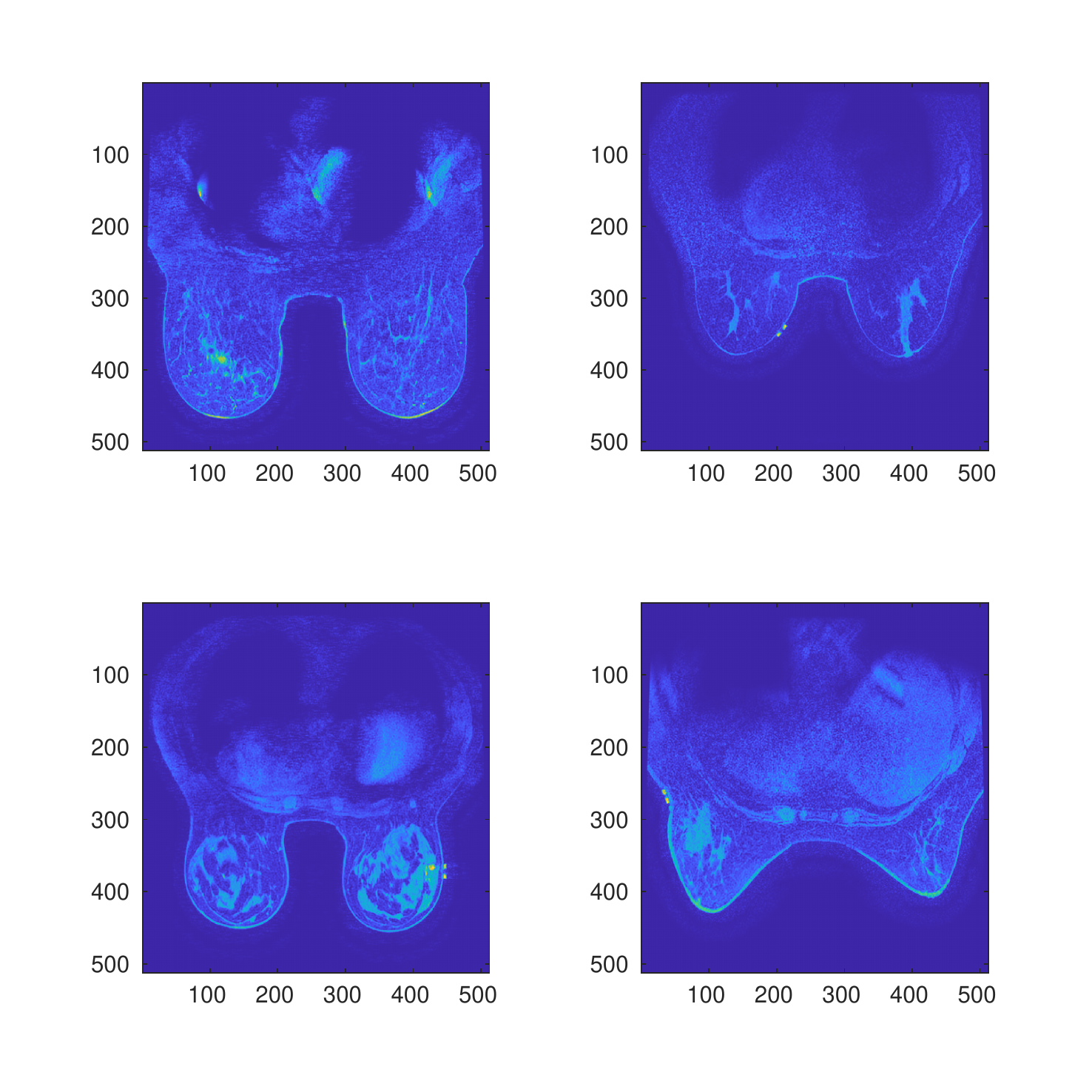}
\caption{Examples of the considered averaged MRI sequences. Depicted are 4 transaxial slices corresponding to 4 different patients. The fields are axis-aligned.}\label{fig:example}
\end{figure}

\subsubsection{The Data}

The dataset comprises a heterogeneous set of MRI sequences of 922 biopsy-confirmed invasive breast cancer patients. We consider the point-wise average of the axis-aligned post-contrast fat-saturated MRI sequences. These MRI images come from scanners with a lot of technical variation stemming from differences in manufacturers, magnetic field strengths and acquisition parameters.

We compute the Super Lifted Euler Characteristic Transform for the MRIs of the 922 patients. Examples of these MRIs are presented in Figure \ref{fig:example}. We restrict the MRI sequences to expert prepared annotation masks, these are depicted in Figure \ref{fig:example2}. The molecular subtypes and markers that we are predicting are presented in Table \ref{table:freqs} together with their frequencies. Further details of the data are presented in Appendix \ref{secA2}.

\begin{figure}[!thbp]
\includegraphics[width=\linewidth]{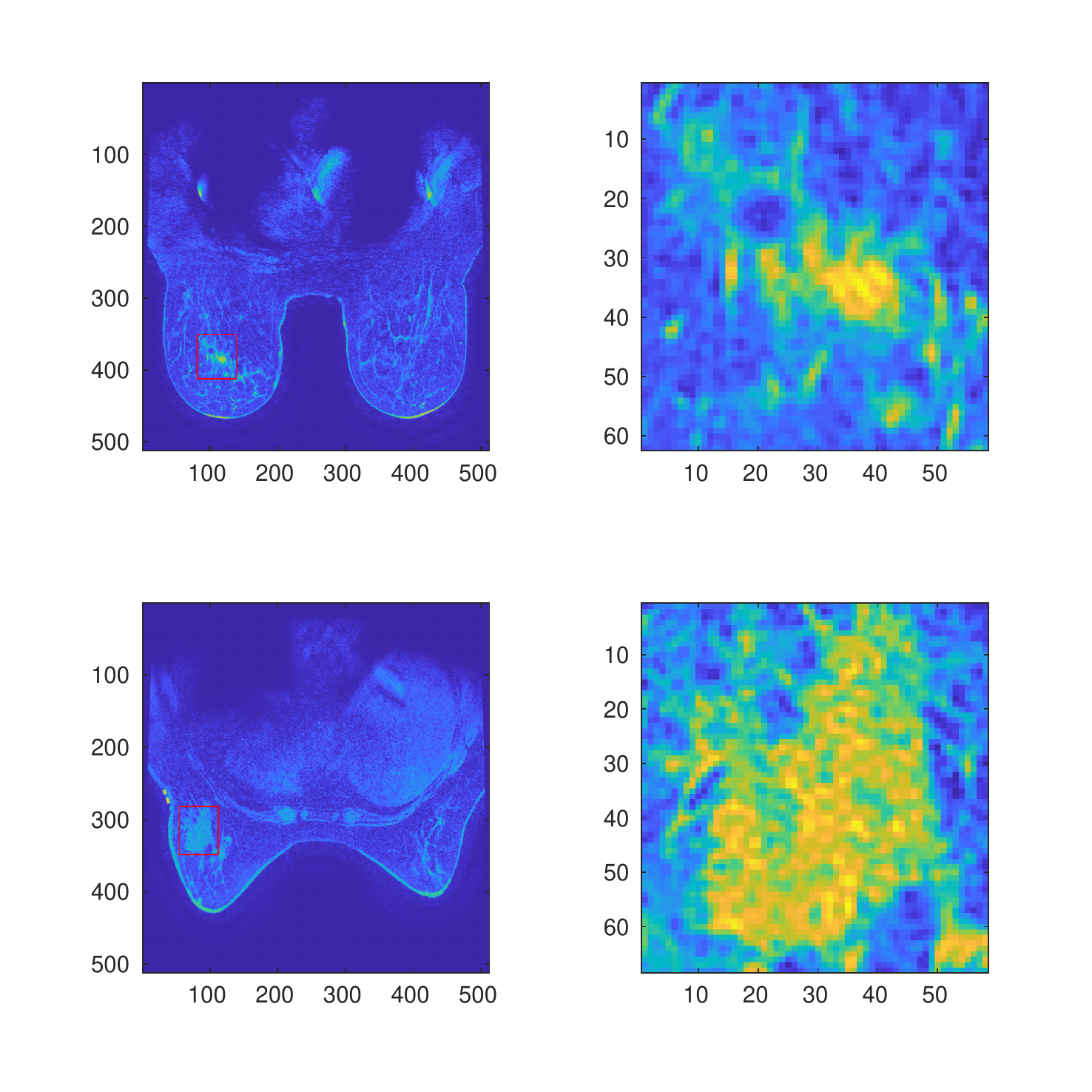}
\caption{Example slices of the considered annotation masks. On Left: transaxial slices of MRI scans. The redboxes denote the annotation masks (which are depicted in detail on right). The top row is from an ER and PR positive, Luminal Subtybe A patient, while the bottom is a Triple Negative Breast Cancer.} \label{fig:example2}
\end{figure}

\begin{table}[h!]
\caption{\label{table:freqs} Frequencies of the molecular markers and subtypes in the data. The top row records the number of patients with the corresponding Marker/Subtype, and the bottom patients without it.}
\begin{tabular}{|c|c|c|c|c|c|c|c|} \hline
Marker & ER + & PR + & HER2 + & Luminal A & Luminal B & HER2 & TNBC \\ \hline
Positive & 686         & 598         & 163           & 104       & 595       & 59   & 163  \\
Negative & 236         & 324         & 759           & 818       & 327       & 863  & 759 \\ \hline
\end{tabular}
\end{table}

\subsubsection{Computing the transform}

The annotation masks of the fields vary in size. To make the transforms comparable, we scale each annotation mask to $[-1,1]^3$ and the field values collectively to $[0,1]^3$. Each field is then supported in the cube $[-1,1]^3$, and the field values are restricted to $[0,1]$. However, most fields do not exhibit such high variation; the values they attain are much more concentrated. A select few fields with their minimum and maximum values are presented in Figure \ref{forestplot}.

As described in the earlier sections, we treat the pixel data from the MRI images as an evaluation of a field on a grid. The superlevel sets of this field can again be extracted using the matlab isosurface procedure. To be able to compute the transform, we need to pick the function value thresholds that we evaluate the transform on.

An obvious idea for selecting the function value thresholds would be to select them uniformly on the interval $[0,1]$, like we did with the simulated dataset. However, as is seen in Figure \ref{forestplot}, the field values are not spread uniformly, they seem to concentrate on the lower end. We want to pick the thresholds where there is most information, therefore we choose a set of loosely exponentially spread thresholds, in the same spirit as in Example \ref{parameterexample}. The distances between the transforms are then obtained through simple numerical integration.

\begin{figure}[!ht]
    \centering
    \includegraphics[width=12cm]{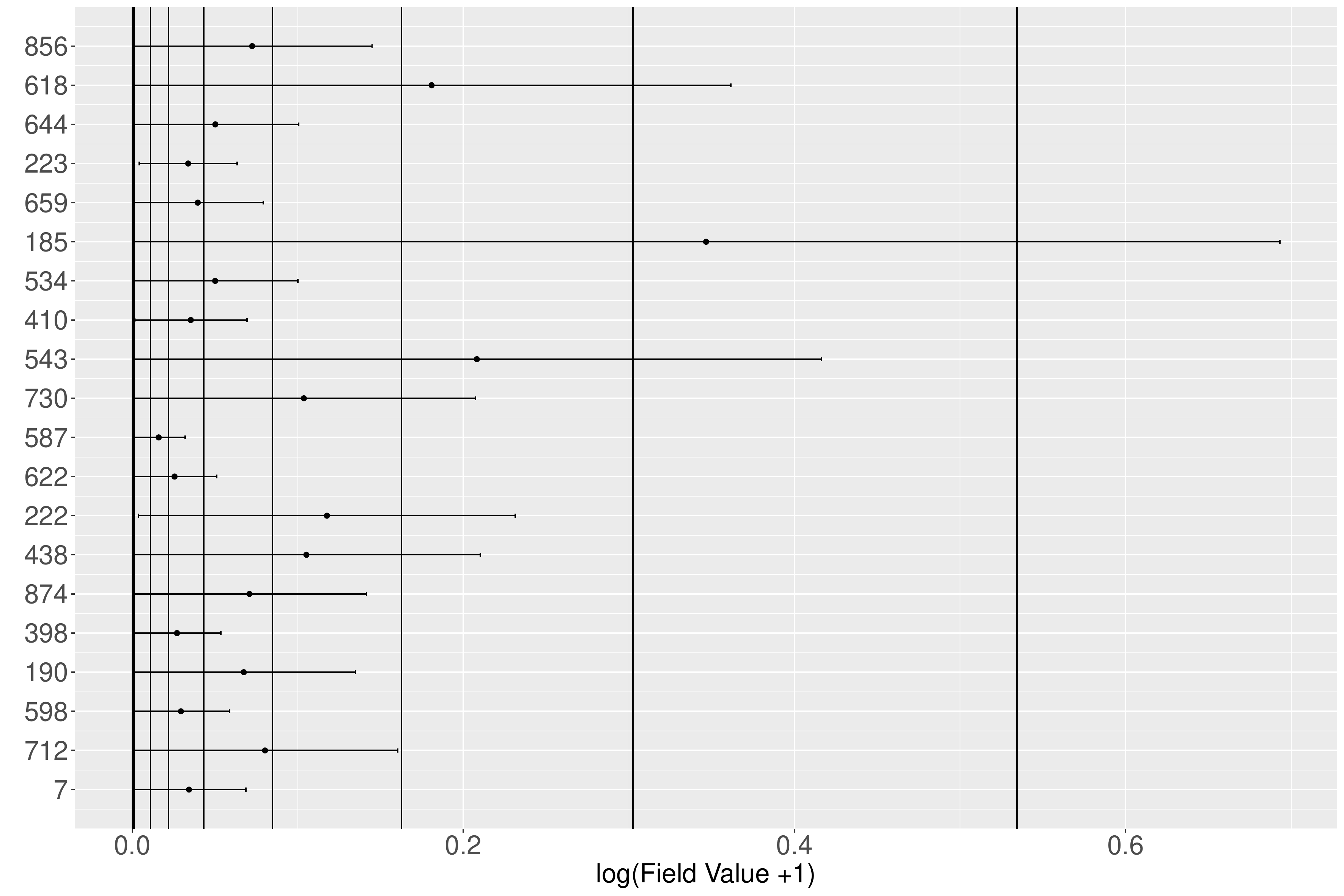}
    \caption{The fields have large differences in their variability. Here are 50 displayed on a logarithmic scale, with the field minimum, maximum and average represented by the horizontal bars. To ensure a non-trivial representation for all of the fields, our field value filtration concentrates mainly on the smaller side. The exact threshold values are \{5,10,100,200,400,800,1600,3200,6400\}/9061. The thresholds are marked by vertical lines.}%
    \label{forestplot}%
\end{figure}

For each superlevel set, we evaluate the Super Lifted Euler Characterstic Transform on a grid of 362 approximately uniformly spread directions, and 100 uniformly spread height filtrations. As we have 9 function value filtrations, each field is then represented as a $362 \times 100 \times 9$ vector.

\subsubsection{Modeling}

To get a meaningful benchmark, we follow the steps taken in \cite{Saha}. This means our prediction metric is Area Under the Curve (AUC), and we use a 50-50 test/train split to determine this metric. We fit a support vector machine with the standard double exponential kernel described, i.e.

$$
K(x_i,x_j) = e^{-\frac{d(x_i,x_j)^2}{\lambda^2}}.
$$

The bandwidth is obtained through the median distance heuristic:

$$
\lambda = \textrm{median}_{i \neq j} d(x_i,x_j).
$$

This is a widely used and generic classification model. As a kernel method it has the advantage that the model only uses an $n$ by $n$ matrix of dissimilarities and an $n$-binary vector as its input. This lets us isolate the data intensive computations from modeling.
For distance we use the distance of the Super Lifted Transforms, and for comparison we also consider the distance obtained from margninalizing the SELECT. 

\subsubsection{Results}

\begin{table}[!htbp] \centering 
  \caption{\label{tableRes}  The Area Under the Curve for the 2 methods. Both use 50-50 train and test split and DeLong's method to compute the confidence intervals.} 
\begin{tabular}{@{\extracolsep{1pt}} cccc|ccc|ccc} 
\\[-1.8ex]\hline 
\hline \\[-1.8ex]
 & \multicolumn{3}{c|}{\cite{Saha}} & \multicolumn{3}{c|}{SELECT} & \multicolumn{3}{c}{Marginal SELECT}  \\
\\[-1.8ex]\hline 
\hline \\[-1.8ex]  
 & 2.5\% & Est. & 97.5\% & 2.5\% & Est. & 97.5\% & 2.5\% & Est. & 97.5\% \\ 
\hline \\[-1.8ex] 
Luminal A & 0.647 & 0.697 & 0.746 & 0.513 & 0.568 & 0.623 & 0.462 & 0.519 & 0.576 \\ 
Luminal B & 0.494 & 0.566 & 0.638 & 0.434 & 0.523 & 0.611 & 0.419 & 0.501 & 0.583 \\ 
HER2 & 0.539 & 0.633 & 0.727 & 0.548 & 0.641 & 0.734 & 0.526 & 0.611 & 0.697 \\ 
TNBC & 0.589 & 0.654 & 0.720 & 0.491 & 0.558 & 0.625 & 0.449 & 0.521 & 0.593 \\ 
ER+ & 0.591 & 0.649 & 0.705 & 0.552 & 0.610 & 0.668 & 0.523 & 0.582 & 0.641 \\ 
PR+ & 0.569 & 0.622 & 0.674 & 0.527 & 0.582 & 0.636 & 0.518 & 0.574 & 0.630 \\ 
HER2+ & 0.433 & 0.500 & 0.567 & 0.464 & 0.535 & 0.605 & 0.532 & 0.596 & 0.660 \\ 
\hline \\[-1.8ex] 
\end{tabular} 

\end{table}

The results and those achieved in \cite{Saha} are presented in Table \ref{tableRes}. We see that SELECT performs similarly to the machine learning methods on this homogeneous set of data. The overall classification rates seem somewhat higher for the machine learning methods, but with the overlapping intervals the evidence is by and large inconclusive. The most notable difference is in predicting the Luminal A subtype, with non-overlapping confidence intervals. Another notable difference is in predicting the TNBC, where the ML methods achieve better results, with minor overlap in the confidence intervals. Contrary to the machine learning methods in \cite{Saha}, we use the same model for all of the tasks, and our classifier is a standard non-tailored classifier with default parameters. We conclude that SELECT can capture the signal despite the large technical variation in the dataset.

When comparing the SELECT distance to the marginal distances, we see that the distance from marginal Euler curves seems to perform worse than the distance computed from the full Transform. Although the differences are not significant, this suggests the full transform may be more informative.

\section{Discussion}

In this paper we have introduced the Lifted Euler Characteristic Transform that generalizes the Euler Characteristic Transform \cite{PHT,CMT} to definable functions. Concretely this allows us to analyze field type of data that do not fit into the shape framework. We have proven that the  transform is injective and a stratified map, properties that are appealing both theoretically and for applications. We have shown that for certain non-trivial Moduli spaces of fields there is an upper bound on the number of directions that determine any field in the space. We have also
demonstrated the practical utility of the transform on simulated and real data.

The ideas presented in this paper point to further research directions of theoretical and applied interest, such as:
\begin{enumerate}
\item Moving from simply regression to subfield selection. The lifted transform preserves the most important theoretical properties of its shape analog. In \cite{SINATRA} a procedure was proposed to extract the three-dimensional coordinates of shapes that are most important for 
differentiating between two classes of shapes, this is the problem of subshape selection. The stratification properties of the ECT were central
in computing a  pullback from the transform to coordinates on a shape; a central step in subimage selection. Adapting the approach to subimage 
selection \cite{SINATRA} to fields would be of great interest.
\item Do marginal Euler curves determine a field? This is certainly true for positive piecewise constant fields by Proposition \ref{prop:weighted}. Could the marginals also be used for feature selection problems?
\item The key viewpoint in this paper is to look at fields as carriers of geometric information as opposed to considering them as weights. The idea of integrating geometric information is known to the experts of algebraic geometry by the name of motivic integration. Is there a motivic construction to the framework we propose? Promising work in this direction was presented in a recent preprint \cite{leb21}.
\item Manifolds decorated with more complicated structures. A function assigns a value to each point on a manifold. Lifting this is straightforward and this allows us to analyze for example MRI images. However, there are imaging modalities where more complicated objects are assigned. One such modality is Diffusion Tensor Imaging, in which a flow matrix is assigned to each point. Could we use a similar lift to study these, and what properties would that transform have?
\end{enumerate}

\section*{Acknowledgments}

The authors would like to acknowledge conversations with Justin Curry, Kate Turner, Rob Ghrist, Robert Adler, Xiaojun Zheng, and Lorin Crawford. HK would like to thank Maciej Mazurowski, Ashibrani Saha, and Marc Ryser for helpful conversations. The authors would also like to thank the anonymous reviewers for their insightful comments.
The authors would like to acknowledge partial funding from HFSP RGP005, NSF DMS 17-13012, NSF BCS 1552848, NSF DBI 1661386, NSF IIS 15-46331, 
NSF DMS 16-13261, as well as high-performance computing partially supported by grant 2016-IDG-1013 from the North Carolina
Biotechnology Center. Any opinions, findings, and conclusions or recommendations expressed in this
material are those of the author(s) and do not necessarily reflect the views of any of the funders.

\appendix

\section{Schapira's inversion formula}\label{secA1}

To make the paper self contained we provide the statement of Schapira's inversion formula as it will be essential in proving injectivity properties 
of the topological transforms we propose in this paper.

Euler calculus comes with a set of canonical operations including pullbacks, pushforwards, and convolution.

\begin{defn}
Let $f: X \to Y$ be a tame mapping between between definable sets.
Let $\phi_Y: Y \to \mathbb{Z}$ be a constructible function on $Y$.
The \define{pullback} of $\phi_Y$ along $f$ is defined pointwise by
\[
f^*\phi_Y(x)=\phi_Y(f(x)).
\]
The pullback operation defines a ring homomorphism $f^*:\CF(Y) \to \CF(X)$.
\end{defn}

The dual operation of pushing forward a constructible function along a tame map is given by integrating along the fibers. 

\begin{defn}
The \define{pushforward} of a constructible function $\phi_X:X \to \mathbb{Z}$ along a tame map $f:X\to Y$ is given by
\[
f_*\phi_X(y)=\int_{f^{-1}(y)} \phi_X d\chi.
\]
This defines a group homomorphism $f_*: \CF(X) \to \CF(Y)$.
\end{defn}

Putting these two operations together allows one to define our first topological transform: the Radon transform.

\begin{defn}
Suppose $S\subset X \times Y$ is a locally closed definable subset of the product of two definable sets.
Let $\pi_X$ and $\pi_Y$ denote the projections from the product onto the indicated factors.
The \define{Radon transform with respect to $S$} is the group homomorphism $\Rad_S : \CF(X) \to \CF(Y)$ that takes a constructible function on $X$, $\phi:X \to \mathbb{Z}$, pulls it back to the product space $X\times Y$, multiplies by the indicator function of $S$ before pushing forward to $Y$. 
In equational form, the Radon transform is
\[
\Rad_S (\phi):=\pi_{Y*} [(\pi_X^*\phi) 1_S].
\]
\end{defn}

The following inversion theorem of Schapira~\cite{Schapira:tom} gives a topological criterion for the invertibility of the transform $\Rad_S$ in terms of the subset $S\subset X \times Y$.

\begin{thm}[\cite{Schapira:tom} Theorem 3.1]\label{thm:inversion}
If $S\subset X\times Y$ and $S'\subset Y\times X$ have fibers $S_x$ and $S'_x$ in $Y$ satisfying
\begin{enumerate}
	\item $\chi(S_x\cap S'_x)=\chi_1$ for all $x\in X$, and
	\item $\chi(S_x\cap S'_{x'})=\chi_2$ for all $x'\neq x \in X$,
\end{enumerate}
then for all $\phi\in \CF(X)$,
\[
(\Rad_{S'} \circ \Rad_{S})\phi = (\chi_1-\chi_2)\phi +\chi_2 \left(\int_X \phi d\chi\right) 1_X .
\]
\end{thm}

\section{Details on the breast cancer MRI dataset}\label{secA2}
In this section we provide further details on the MRI application.
\subsection{The Dataset}
The dataset contains imaging data for 922 female patients observed from Jan 1, 2000 to Mar 23, 2014 with invasive breast cancer. 

The dataset includes information on three molecular markers: The Estrogen receptor (ER+), progesterone receptor (PR+) and the human epidermal growth factor receptor 2 (HER2+). These are binary labels indicating presence of the corresponding receptor, which is obtained by thresholding Allred score based on an immunohistochemistry analysis. From the three binary labels, the four cancer subtypes were defined using a standard identification rule. The identification rule is presented in Table \ref{tbl:markers}.

The seven classification tasks corresponding to the 3 markers and 4 subtypes are the same tasks as in \cite{Saha}. In that work, the authors also considered the proliferation marker Ki-67, but as this cannot be inferred from the publically available dataset, we omit it from our analyses.

\begin{table}[!h]
\centering
\caption{Determing the cancer subtype by the presence of molecular markers. 1 denotes positive status, 0 negative.}\label{tbl:markers}
\begin{tabular}{c|ccc}
                              & \multicolumn{3}{c}{Markers}                             \\ \hline
Subtype                       & \multicolumn{1}{c|}{ER+} & \multicolumn{1}{c|}{PR+} & HER2+ \\ \hline %\\[-1.8ex]\hline 
Luminal B                     & \multicolumn{1}{c|}{1}  & \multicolumn{1}{c|}{1}  & 1    \\ \hline
Luminal A                     & \multicolumn{1}{c|}{1}  & \multicolumn{1}{c|}{1}  & 0    \\ \hline
Luminal B                     & \multicolumn{1}{c|}{1}  & \multicolumn{1}{c|}{0}  & 1    \\ \hline
Luminal B                     & \multicolumn{1}{c|}{0}  & \multicolumn{1}{c|}{1}  & 1    \\ \hline
Luminal A                     & \multicolumn{1}{c|}{1}  & \multicolumn{1}{c|}{0}  & 0    \\ \hline
Luminal A                     & \multicolumn{1}{c|}{0}  & \multicolumn{1}{c|}{1}  & 0    \\ \hline
HER2                          & \multicolumn{1}{c|}{0}  & \multicolumn{1}{c|}{0}  & 1    \\ \hline
Triple Negative Breast Cancer & \multicolumn{1}{c|}{0}  & \multicolumn{1}{c|}{0}  & 0    \\ \hline
\end{tabular}
\end{table}

\subsection{Data Processing}

Each patient has multiple post-contrast fat-saturated MRI sequences. The number of sequences varies from around 3 to 5 per patient and there is no correspondence between the sequences. In order to obtain a single field for each patient, we use pointwise average of the post-contrast MRI sequences. An illustration of this procedure is given in Figure \ref{AVERAGES}. We see the differences between the sequences are small, so the mean aggregation is likely sufficiently expressive compared to a more complicated model that considers all the sequences as separate inputs.

\begin{figure}[h]
\centering
\includegraphics[width=12cm]{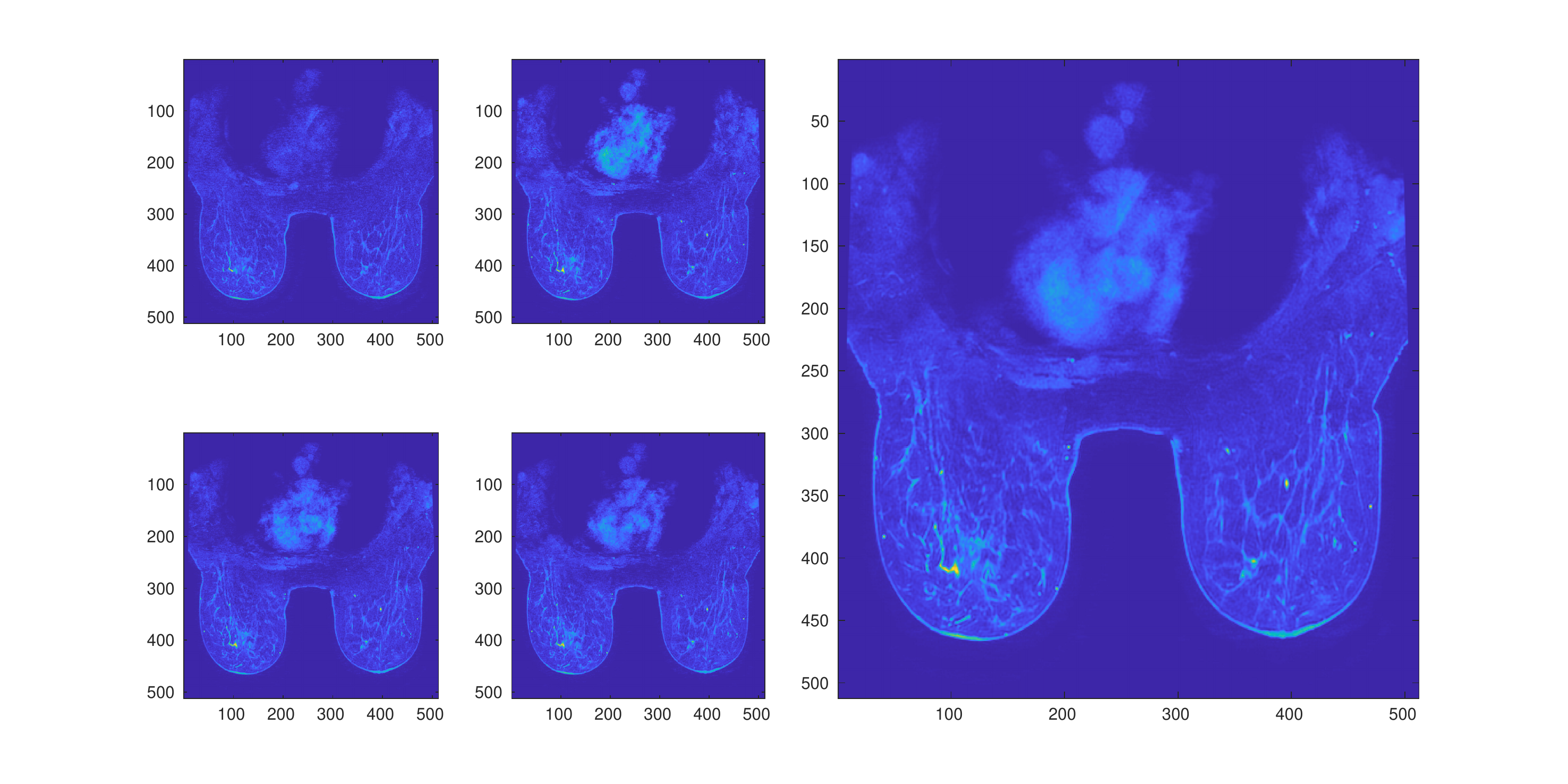}
\caption{On left: The raw post-contrast sequences. Right: The pointwise average used in the analysis.}\label{AVERAGES}
\end{figure}

\subsection{Details on the approach taken in \cite{Saha}}

The models considered in \cite{Saha} are machine learning models fitted on 529 imaging features extracted based on radiomics and computer-aided diagnosis. These features are built on an extensive amount of clinical research and they include, among others: Breast and fibrograndular tissue (FGT) volume features, Tumor size and morphology, tumor texture, FGT, FGT enhancement, tumor enhancement and combinations thereof. The exact set of features is detailed in the supplementary material of \cite{Saha}.

A separate model was fitted for each of the tasks. The models started with $N$ features that were selected based on a univariate model that maximizes the AUC curve for predicting the label. Among these $N$ features, highly correlated features were removed from the model, and a random forest was fitted based on the remaining features. The random forest hyperparameters, $N$ and the correlation threshold were optimized using cross validation maximizing the Area Under the Curve.

\end{document}